\documentclass{amsart}
\usepackage[latin9]{inputenc}
\usepackage{verbatim}
\usepackage[british]{babel}
\usepackage{url}
\usepackage{amstext}
\usepackage{amsthm}
\usepackage{amssymb,xcolor,MnSymbol,graphicx}
\usepackage[unicode=true,
 bookmarks=true,bookmarksnumbered=false,bookmarksopen=false,
 breaklinks=false,pdfborder={0 0 1},backref=false,colorlinks=false]
 {hyperref}

\makeatletter
\theoremstyle{plain}
\newtheorem{thm}{\protect\theoremname}
  \newtheorem{lem}[thm]{\protect\lemmaname}
  \newtheorem{prop}[thm]{\protect\propositionname}
  
\theoremstyle{definition}
 \newtheorem{defn}[thm]{\protect\definitionname}
 \newtheorem{example}[thm]{\protect\examplename}
\newtheorem{assumption}[thm]{Assumption}
\theoremstyle{remark}
\newtheorem{remark}[thm]{Remark}
\usepackage[shortlabels]{enumitem}
\setlist{leftmargin=*}

\makeatother
 
  \providecommand{\lemmaname}{Lemma}
  \providecommand{\propositionname}{Proposition}
\providecommand{\theoremname}{Theorem}
\providecommand{\corollaryname}{Corollary}
\providecommand{\definitionname}{Definition}
\providecommand{\examplename}{Example}

\let\div\relax
\DeclareMathOperator*{\div}{div}
\DeclareMathOperator*{\curl}{curl}

\DeclareMathOperator*{\esssup}{ess\,sup}
\DeclareMathOperator*{\supp}{supp}

\newcommand{\norm}[1]{\left\Vert#1\right\Vert}
\newcommand{\psic}[1]{\psi_{#1}}
\newcommand{\nuc}[1]{\nu_{#1}}
\newcommand{\partialc}[1]{\partial_{#1}}
 \newcommand{\psich}[1]{{\hat{\psi}}_{#1}}
 \newcommand{\f}[1]{\hat {#1}}

\begin{document}
\def\C{\mathbb{C}}
\def\R{\mathbb{R}}
\def\N{\mathbb{N}}
\def\Hcurl{{\mathcal{H}(\curl,\Omega)}}
\def\Hcurlzero{\mathcal{H}(\curl 0,\Omega)}
\def\Hdotcurl{\mathcal{\dot H}(\curl,\Omega)}
\def\Hocurl{\mathcal{H}_{0}(\curl,\Omega)}
\def\Hdotocurl{\mathcal{\dot H}_{0}(\curl,\Omega)}
\def\Hdiv{\mathcal{H}(\div,\Omega)}
\def\Hdotdiv{\mathcal{\dot H}(\div,\Omega)}
\def\Calpha{C^{0,\alpha}(\overline{\Omega};C^{3})}
\def\bo{\partial\Omega}
\def\Hmone{{\mathcal{H}}^{-1} (\Omega)}
\def\Hdivzero{\mathcal{H}(\div 0,\Omega)}
\def\Hdotdivzero{\mathcal{\dot H}(\div 0,\Omega)}
\def\Hdivzerozero{\mathcal{H}_0(\div 0,\Omega)}
\def\Hdotdivzerozero{\mathcal{\dot H}_0(\div 0,\Omega)}
\def\Honetzero{{\mathcal{H}}^{1}_{t0} (\Omega)}
\def\Honed{\dot{\mathcal{H}}^1 (R^3)}
\def\Hone{\dot{\mathcal{H}}^1(B_R) }
\def\Hdone{\mathcal{H}}
\def\Honeomega{\mathcal{H}(\Omega)}
\def\I{\mathcal{I}}
\def\Honesq{ \mathcal{H}_1 }
\def\dom{\mathrm{Dom}}
\def\ran{\mathrm{Rang}}
\def\ker{\mathrm{Ker}}
\def\kn{K_N(\Omega)}
\def\kt{K_T(\Omega)}

\newcommand{\hookdownarrow}{\mathrel{\rotatebox[origin=c]{-90}{$\hookrightarrow$}}}

\renewcommand{\epsilon}{\varepsilon}

\title{Essential spectrum for Maxwell's equations}

\author[G.S. Alberti]{Giovanni S. Alberti}
\address{Department of Mathematics, University of Genoa, Via Dodecaneso 35, 16146 Genoa, Italy.}
\email{alberti@dima.unige.it}

\author[M. Brown]{Malcolm Brown}
\address{Cardiff School of Computer Science and Informatics, Queens Buildings, 5 The Parade, Cardiff, CF24 3AA, UK.}
\email{brownbm@cardiff.ac.uk}

\author[M. Marletta]{Marco Marletta}
\address{Cardiff School of Mathematics, 21-23 Senghennydd Road, Cardiff, CF24 4AG, UK.}
\email{marlettam@cardiff.ac.uk}

\author[I. Wood]{Ian Wood}
\address{School of Mathematics, Statistics and
Actuarial Science,
University of Kent,
Sibson Building,
Canterbury CT2 7FS, UK.}
\email{I.Wood@kent.ac.uk}

\thanks{{\bf Acknowledgements:} The authors express their sincere thanks to Dr Pedro Caro of BCAM, who visited us on several occasions
and provided a lot of helpful comments and useful insights. The authors are also very grateful to the 
UK Engineering and Physical Sciences Research Council for support under grant EP/K024078/1 and for the funding of the LMS and EPSRC
(grant EP/K040154/1) which enabled them to attend the LMS-EPSRC Durham Symposium on {\em Mathematical and Computational Aspects of Maxwell's Equations}. }

\begin{abstract}
We study the essential spectrum of operator pencils associated with anisotropic Maxwell equations, with permittivity $\epsilon$, permeability $\mu$ and conductivity $\sigma$, on finitely connected unbounded domains. The main result is that the essential spectrum of the Maxwell pencil is the union of two 
sets: namely, the spectrum of the pencil $\div((\omega\epsilon + i \sigma) \nabla\,\cdot\,)$, and the essential spectrum of the Maxwell pencil with constant coefficients. 
We expect the analysis to be of more general interest and to open avenues to investigation of other questions concerning Maxwell's and related systems.
\end{abstract}

\date{24th April 2018}

\maketitle

\section{Introduction}
In this paper we consider the essential spectrum of linear operator pencils arising from the Maxwell system
\begin{equation}
\left\{ \begin{array}{ll}
\curl H=-i (\omega\varepsilon+i\sigma)E& \text{in $\Omega$,}\\
\curl E=i\omega\mu H & \text{in $\Omega$,}\\

\end{array} \right.  \label{eq:maxwell}
\end{equation}
where  $\Omega\subseteq {\mathbb R}^3$ is a finitely connected domain, with boundary condition 
\[ \mbox{ $ \nu\times E = 0$ on $\partial \Omega$} \]
if $\Omega$ has a boundary. In these equations  $\omega$ is the pencil spectral parameter, $\epsilon$ the electric permittivity, $\mu$ the magnetic
permeability and $\sigma$ is the conductivity; $\nu$ is the unit normal to the boundary. 

Lassas \cite{lassas-1998} already studied this problem on a bounded
domain with $C^{1,1}$ boundary so in this article our primary concern is to treat unbounded
domains which provides additional sources for essential spectrum. However, even for bounded domains, we are able to relax the required boundary regularity to Lipschitz continuity. Like Lassas we allow the permittivity, permeability and conductivity to be tensor valued (i.e.\ we allow anisotropy);
however we make the physically realistic assumption that, at infinity, these coefficients approach isotropic constant values. 

Maxwell systems in infinite domains are usually studied in the context of scattering, with a Silver-M\"{u}ller radiation condition
imposed at infinity, see, e.g.\ \cite[p. 10]{MONK-2003} and \cite{cakoni-2011,cakoni-2014}. Scattering theory is sometimes regarded as the study of solutions 
when the spectral parameter lies in the essential spectrum, though the fact that the Maxwell system already has non-trivial
essential spectrum in bounded domains indicates that such an interpretation involves local conditions as well
as the study of radiation to infinity. The case of zero conductivity $\sigma\equiv 0$ is substantially simpler, both for bounded 
and unbounded domains. However it is also physically unrealistic in numerous applications, including imaging 
\cite{MR2581979,MR2783934,MR2294900,MR2296930}.

The main technical difficulty in dealing with the essential spectrum of Maxwell systems in infinite domains is 
the fact that compactly supported perturbations to the coefficients do change the essential spectrum, as is clear even
for bounded domains from \cite{lassas-1998}. This means that techniques such as Glazman decomposition, useful for 
Schr\"{o}dinger operators, are no longer helpful. We use instead a Helmholtz decomposition
inspired by \cite{alberti-capdeboscq-2016,alberti-maxwell-holder}  together with further decompositions of the resulting $2\times 2$ block operator matrices. As in \cite{AC2014}, this approach allows us to substantially reduce Maxwell's system to an elliptic problem.
The main result is stated in Theorem \ref{thm:adv}: the essential spectrum of the Maxwell pencil is the union of two 
sets: namely, the spectrum of the pencil $\div((\omega\epsilon + i \sigma) \nabla\,\cdot\,)$ acting between suitable
spaces, together with the essential spectrum of the Maxwell pencil with constant coefficients. The 
spectral geometric question of how the topology of $\Omega$ at infinity is reflected in the essential spectrum of a constant coefficient
Maxwell operator is also interesting, and an avenue for future work.

{Our original motivation for the investigations in this paper came from our study of inverse problems in a slab for the Maxwell system with conductivity.
However a knowledge of the essential spectrum has much more fundamental importance. It is a first step towards determination of the absolutely continuous subspace 
of an operator and hence the behaviour of its semi-group, as required, e.g., for the study of Vlasov-Maxwell systems. It can also be a key component in the analysis of
certain types of homogenisation problem.}

\section{Main result}
We shall study the Maxwell system on a finitely-connected domain $\Omega\subseteq\mathbb{R}^{3}$.
Prototype examples include exterior domains $\Omega:=\mathbb{R}^{3}\setminus \overline{\Omega'}$ in which $\Omega'$ 
has finitely many simply connected components; the case of an infinite slab, $\Omega = \R^2\times (0,1)$, or
a half-space $\Omega = \R^2\times (0,\infty)$; domains with cylindrical ends, such as waveguides; and indeed the case $\Omega=\R^3$ (see Assumption~\ref{mm:HD} and Proposition~\ref{prop:mm9} below for more details).  
The boundary $\bo$, if non-empty,  
will be of  Lipschitz type, and the coefficients $\varepsilon$, $\sigma$ and $\mu$ will be assumed to lie in
$L^{\infty}\left(\Omega;\mathbb{R}^{3\times3}\right)$ and be such that for some $\Lambda>0$ and every $\eta\in\mathbb{\mathbb{R}}^{3}$ 
\begin{equation}
\Lambda^{-1}
   |\eta |^{2}\le\eta\cdot\varepsilon\eta\le\Lambda|\eta|^{2},\quad\Lambda^{-1}  |  \eta |^{2}
\le\eta\cdot\mu\eta\le\Lambda|\eta|^{2},\quad0\le\eta\cdot\sigma\eta\le\Lambda|\eta|^{2}\label{eq:Hyp-Ellip}
\end{equation}
almost everywhere in $\Omega$.

As already mentioned, the case of bounded domains was treated by Lassas \cite{lassas-1998} under slightly stronger regularity assumptions; for infinite domains we assume that all the coefficients 
have a `value at infinity' in the precise sense that
\begin{equation}\label{eq:vai}
\lim_{x\to\infty} \mu(x)=\mu_0 I,\qquad
\lim_{x\to\infty} \epsilon(x)=\epsilon_0 I,\qquad
\lim_{x\to\infty} \sigma(x)=\sigma_0 I,\qquad
\end{equation}
for some scalar values $\mu_0>0$, $\epsilon_0>0$ and $\sigma_0\geq 0$. To allow a unified treatment of unbounded and bounded domains, it is
convenient to assign values to $\epsilon_0$, $\mu_0$ and $\sigma_0$ when $\Omega$ is bounded, and we choose
\begin{equation} \label{eq:vaib}\epsilon_0:= 1, \;\; \mu_0 := 1; \;\; \sigma_0 := 0, \;\;\; \mbox{($\Omega$ bounded).} \end{equation}

Several function spaces arise commonly in the study of Maxwell systems; to fix notation,
we denote
\begin{align*} 
   \Hcurl:={}& \{ u\in L^2(\Omega;\C^3 ) :  \curl u  \in L^2(\Omega;\C^3 )     \},  \\  
  \Hdiv:={}& \{ u\in L^2(\Omega;\C^3 ) : \div  u  \in L^2(\Omega)    \}. 
\end{align*}
If $\bo$ is non-empty then we let $\nu$ denote the outward unit normal vector, and define
\[
\Hocurl= \{ u\in \Hcurl :  \nu\times u|_{\bo} =0 \}, 
\]
with the understanding that when $\Omega=\R^3$ then  $\Hocurl = \Hcurl  $.

We start by considering, in the Hilbert space
\begin{equation}
 \Honesq:=   \Hocurl  \oplus    \Hcurl ,   \label{eq:honesq}
 \end{equation}
the operator pencil $\omega\mapsto V_\omega$ defined from (\ref{eq:maxwell})  in the space   
$\Honesq$  by
\begin{equation}\label{eq:maxwell_operator} 
\begin{aligned}
 & V_{\omega}\colon \Honesq\longrightarrow L^{2}(\Omega;\C^{3})^{2},\\
 & (E,H)\longmapsto(\curl H+i(\omega\varepsilon+i\sigma)E,\curl E-i\omega\mu H).   
\end{aligned}
\end{equation}
Our aim is to study the essential spectrum of the pencil $V_\omega$.

\begin{defn}\label{def1}
Let $H_1$ and $H_2$ be two Hilbert spaces. 
For each $\omega\in\C$, let $L_\omega:H_1\to H_2$ be a bounded linear operator. 
We say that $\omega$ lies in the essential spectrum of the pencil $\omega\mapsto L_\omega$  if $0$ lies in the $\sigma_{e,2}$ essential spectrum
of the operator $L_\omega$ as defined in \cite[Ch.\ I, \S 4]{EE87}; explicitly, if $L_\omega$ is not in the class ${\mathcal F}_+$ of semi-Fredholm operators
with finite-dimensional kernel. 
\end{defn}

\begin{remark}\label{rem:essential}\hspace{1cm}
\begin{enumerate}[(a)]
\item \label{r2a} By  \cite[Ch.\ I, Cor.\ 4.7]{EE87}, the statement `$0$ lies in $\sigma_{e,2}(L_\omega)$' is equivalent to the statement that there exists 
a {\em Weyl singular sequence} $(u_n)$ in $H_1$ with
$\| u_n \|_{H_1}=1$ and $ u_n \rightharpoonup 0$ in $H_1$ such that $|| L_\omega u_n ||_{H_2}  \to  0$. 
\item We shall often abuse terminology and say `$\omega$ lies in the essential spectrum of $L_\omega$' or write `$\omega\in\sigma_{ess}(L_\omega)$'. 
\item \label{r2c} In our situation we deal exclusively with densely defined operators having closed range. By \cite[Ch.\ I, Thm.\ 3.7]{EE87} such operators
are semi-Fredholm with finite-dimensional kernel if and only if they are Fredholm; in the terminology of \cite[Ch.\ I, \S 4]{EE87}, the essential spectra
$\sigma_{e,2}$, $\sigma_{e,3}$ and $\sigma_{e,4}$ coincide and we have the following equivalent characterisations of the essential spectrum:
\[
\begin{split}
\omega\in\sigma_{ess}(L_\omega) &\Longleftrightarrow L_\omega \hbox{ is not Fredholm } \\
& \Longleftrightarrow  \hbox{ for every compact  } K, \ L_\omega-K \hbox{ is not invertible. } 
\end{split}
\]
The last equivalence is a classical characterisation of Fredholm operators, see, e.g., Kantorovich and Akilov \cite[Chapter XIII, \S 5]{KA}.
\end{enumerate}
\end{remark}
Finally, we introduce some homogeneous
Sobolev spaces which are required for the Helmholtz decomposition for unbounded domains. For bounded domains these coincide with the usual Sobolev
spaces. 

\begin{defn}\label{mm:beppo}\hspace{1cm}

 (1) ($\Omega$ unbounded)
The homogeneous Sobolev spaces $\dot{H}^1_0(\Omega)$ and $\dot{H}^1(\Omega)$ are the completions of the Schwartz spaces ${\mathcal D}(\Omega)$ 
and ${\mathcal D}(\overline{\Omega})$, respectively, with respect to the norm $\| u \| := \| \nabla u \|_{L^2(\Omega)}$.

(2) ($\Omega$ bounded) In this case we define the homogeneous Sobolev spaces to coincide with the usual Sobolev spaces: $\dot{H}^1_0(\Omega) = H^1_0(\Omega)$ 
and $\dot{H}^1(\Omega) = H^1(\Omega)$.
\end{defn}

\begin{remark} 
\hspace{1cm}
\begin{enumerate}[(a)]
\item
Note that this definition does not coincide with Definition 1.31 in \cite{AMROUCHE-BERNARDI-DAUGE-GIRAULT-1998}, which uses Fourier transforms to define
$\dot{H}^s({\mathbb R}^d)$ and results in spaces which are not complete if $s \geq d/2$. Our definition follows Dautray and Lions \cite{dautray-lions-3}. For clarity, we use our definition
directly in Appendix~\ref{sec:appendixA} below.
\item
If $K$ is any compact subset of $\Omega$ with non-empty interior and $\Omega$ is bounded, then the usual $H^1$-norm is equivalent to
\begin{equation} \| u \|^2 := \| u \|_{L^2(K)}^2 + \| \nabla u \|_{L^2(\Omega)}^2, \label{mm:newnorm} \end{equation}
see Maz'ya \cite{mazya-2011}. In the case when $\Omega$ is unbounded, the norms on $\dot{H}^1$ and $\dot{H}^1_0$ may be shown to be
equivalent to the norm defined in (\ref{mm:newnorm}), for any compact $K\subset\Omega$ with non-empty interior. Thus an equivalent definition 
of $\dot{H}^1(\Omega)$, valid for bounded and unbounded $\Omega$, is the closure of ${\mathcal D}(\Omega)$ in the norm (\ref{mm:newnorm}).
However for unbounded $\Omega$ this is no longer equivalent to the $H^1$-norm; e.g. the function given in polar coordinates by $u(r) = 1/(r+1)^{3/2}$ 
does not lie in $H^1({\mathbb R}^3)$ but lies in $\dot{H}^1({\mathbb R}^3)$.
\end{enumerate}
\end{remark}

We are now ready to state our main result.
\begin{thm} \label{thm:adv}
Let $\Omega\subseteq\R^3$ be a Lipschitz domain satisfying Assumption~\ref{mm:HD} (given below) {and $\epsilon,\sigma,\mu\in L^\infty(\Omega;\R^{3\times 3})$ satisfy \eqref{eq:Hyp-Ellip}, \eqref{eq:vai} if $\Omega$ is unbounded, and \eqref{eq:vaib}} if $\Omega$ is bounded. We have
\[
\sigma_{ess}(V_\omega) = \sigma_{ess}\bigl(\div((\omega\epsilon+i\sigma) \nabla\,\cdot\, )\bigr)\, \bigcup \,\sigma_{ess}(V^0_\omega),
\]
where $\div((\omega\epsilon+i\sigma) \nabla\,\cdot\, )$ acts from $\dot{H}^1_0(\Omega;\C)$ to {its dual} $\dot{H}^{-1}(\Omega;\C)$ and 
$V^0_\omega$ is the Maxwell pencil with constant coefficients $\epsilon_0$, $\mu_0$ and $\sigma_0$.
\end{thm}

{
Thanks to this result, the essential spectrum of the Maxwell pencil is decomposed into two parts.
\begin{itemize}
\item The essential spectrum of the operator $\div((\omega\epsilon+i\sigma) \nabla\,\cdot\, )$: this component  depends on the coefficients $\epsilon$ and $\sigma$ directly. In particular, in the case when the coefficients $\epsilon$ and $\sigma$ are continuous, 
it consists of the closure of the set of $\omega=i\nu$, $\nu\in\mathbb R$, for which $\nu\varepsilon+\sigma$ is indefinite at some point in $\Omega$:  see 
Proposition~\ref{prop:mm164}.
\item The essential spectrum of the constant coefficient Maxwell pencil: this component  is related to the geometry of $\Omega$, and depends on the coefficients only through their values at infinity. It can be computed explicitly in many cases of interest: we provide  several examples below.
\end{itemize}

In the next examples we will calculate the essential spectrum of $V_\omega^0$, where
\[
V_\omega^0(E,H)= (\curl H+i(\omega\varepsilon_0+i\sigma_0)E,\curl E-i\omega\mu_0 H),
\]
for different choices of domains $\Omega$.

\begin{example}\label{ex:bounded}
The simplest case to consider in the calculation of $\sigma_{ess}(V_\omega^0)$ is when $\Omega$ is bounded. By \eqref{eq:vaib} we have $\epsilon_0=\mu_0=1$ and $\sigma_0=0$. Thus, the pencil is self-adjoint, and we have
\[
\sigma_{ess}(V_\omega^0) = \{0\},
\]
see \cite{lassas-1998,MONK-2003,kirsch-hettlich-2015}. 
\end{example}
}

\begin{example}\label{ex:full}
We consider here the case of the full space $\Omega=\R^3$. 
We can make use of the Fourier transform to obtain a simple expression of this operator. Writing
$
 E(x) = \int_{\R^3}\hat E(\xi) e^{i x\cdot\xi}\,d\xi,
$
the expression of the operator $\curl E$ in the Fourier domain is given by the multiplication operator $iC(\xi)\hat E(\xi)$, where
\begin{equation}\label{eq:C(xi)}
C(\xi) = \begin{pmatrix}
0 & -\xi_3 & \xi_2 \\
\xi_3 & 0 & -\xi_1 \\
-\xi_2 & \xi_1 & 0 
\end{pmatrix}.
\end{equation}
Writing $\curl H$ in a similar way, we immediately see that $V_\omega^0$ is represented, in the Fourier domain, by the multiplication by the matrix
\[
A_\omega(\xi)=\begin{pmatrix}
 i(\omega\epsilon_0+i\sigma_0)I& iC(\xi) \\
iC(\xi) & -i\omega\mu_0 I
\end{pmatrix}.
\]
A direct calculation gives
\[
\det\left(A_\omega(\xi)\right)=k_\omega(|\xi|^2-k_\omega)^2,\qquad k_\omega=\omega\mu_0(\omega\epsilon_0+i\sigma_0).
\]
By a standard argument, we obtain  that $\sigma_{ess}(V_\omega^0)=\{\omega\in\C:\det(A_\omega(\xi))=0 \text{ for some $\xi\in\R^3$}\}$, so that
\[
\sigma_{ess}(V_\omega^0)=\{\omega\in\C: k_\omega\ge 0\}.
\]
In the particular case when the conductivity at infinity is zero, i.e.\ $\sigma_0=0$, we simply have $\sigma_{ess}(V_\omega^0)=\R$.
\end{example}

\begin{example}
Let us look at the case of the slab $\Omega = \{x=(x',x_3)\in\R^3:0<x_3<L\}$, for some $L>0$. The derivation is very similar to the one presented above for the full space, the only difference being that the continuous Fourier transform in the third variable becomes a Fourier series. As a consequence, the continuous variable $\xi_3$ is replaced by a discrete variable $n=0,1,\dots$. More precisely, 
\begin{align*}
E_j(x) &= \sum_{n=1}^\infty \int_{\R^2}\hat E_j(\xi',n) e^{ix'\cdot\xi'}\sin(\frac{n\pi}{L} x_3) d\xi' ,\qquad j=1,2,\\
E_3(x)& = \sum_{n=0}^\infty \int_{\R^2}\hat E_3(\xi',n) e^{ix'\cdot\xi'}\cos(\frac{n\pi}{L} x_3) d\xi' ,
\end{align*}
and, analogously,
\begin{align*}
H_j(x) &= \sum_{n=0}^\infty \int_{\R^2}\hat H_j(\xi',n) e^{ix'\cdot\xi'}\cos(\frac{n\pi}{L} x_3) d\xi' ,\qquad j=1,2,\\
H_3(x)& = \sum_{n=1}^\infty \int_{\R^2}\hat H_3(\xi',n) e^{ix'\cdot\xi'}\sin(\frac{n\pi}{L} x_3) d\xi' ;
\end{align*}
the range of $n$ in each summation has been determined by the boundary conditions on $x_3=0$ and $x_3=L$. Compared to the full space in Example~\ref{ex:full},
the continuous frequency variable $\xi\in\R^3$ has become $\xi:=(\xi',\frac{n\pi}{L})\in\R^2\times(\frac{\pi}{L}\N)$. By calculations similar to those for the full space,
we see that the essential spectrum is the set of $\omega\in\mathbb C$ such that for some $\xi \in\R^2\times(\frac{\pi}{L}\N)$
\[ k_\omega(|\xi|^2-k_\omega)^2 = 0, \qquad k_\omega=\omega\mu_0(\omega\epsilon_0+i\sigma_0); \]
and it is easy to see that this coincides with the essential spectrum for the full space problem. 
\end{example}

\begin{example}
{
We now compute the essential spectrum of
$
V_\omega^0
$
in a cylinder $\Omega = \{x\in\R^3:0<x_2<L_1,\,0<x_3<L_2\}$. As above, let us expand $E$ and $H$ in Fourier coordinates as
\begin{align*}
E_1(x_1,x_2,x_3) & = \sum_{n\in\N^2}\int_\R \hat E_1(n,\xi)\sin\Bigl(\frac{\pi n_1}{L_1}x_2\Bigr)\sin\Bigl(\frac{\pi n_2}{L_2}x_3\Bigr)e^{i\xi x_1}\,d\xi, \\
E_2(x_1,x_2,x_3) & = \sum_{n\in\N^2}\int_\R \hat E_2(n,\xi)\cos\Bigl(\frac{\pi n_1}{L_1}x_2\Bigr)\sin\Bigl(\frac{\pi n_2}{L_2}x_3\Bigr)e^{i\xi x_1}\,d\xi, \\
E_3(x_1,x_2,x_3) & = \sum_{n\in\N^2}\int_\R \hat E_3(n,\xi)\sin\Bigl(\frac{\pi n_1}{L_1}x_2\Bigr)\cos\Bigl(\frac{\pi n_2}{L_2}x_3\Bigr)e^{i\xi x_1}\,d\xi
\end{align*}
and
\begin{align*}
H_1(x_1,x_2,x_3) & = \sum_{n\in\N^2}\int_\R \hat H_1(n,\xi)\cos\Bigl(\frac{\pi n_1}{L_1}x_2\Bigr)\cos\Bigl(\frac{\pi n_2}{L_2}x_3\Bigr)e^{i\xi x_1}\,d\xi, \\
H_2(x_1,x_2,x_3) & = \sum_{n\in\N^2}\int_\R \hat H_2(n,\xi)\sin\Bigl(\frac{\pi n_1}{L_1}x_2\Bigr)\cos\Bigl(\frac{\pi n_2}{L_2}x_3\Bigr)e^{i\xi x_1}\,d\xi, \\
H_3(x_1,x_2,x_3) & = \sum_{n\in\N^2}\int_\R \hat H_3(n,\xi)\cos\Bigl(\frac{\pi n_1}{L_1}x_2\Bigr)\sin\Bigl(\frac{\pi n_2}{L_2}x_3\Bigr)e^{i\xi x_1}\,d\xi.
\end{align*}
In order to guarantee uniqueness of the expansions, set
\begin{equation}\label{eq:constraints}
\begin{aligned}
\hat H_2(0,n_2,\xi)&=0,&
\hat H_3(n_1,0,\xi)&=0,&
\hat E_1(0,n_2,\xi)&=0,\\
\hat E_1(n_1,0,\xi)&=0,&
\hat E_2(n_1,0,\xi)&=0,&
\hat E_3(0,n_2,\xi)&=0,
\end{aligned}
\end{equation}
for every $n\in\N^2$ and $\xi\in\R$. 
}

{
A direct calculation gives that  the operators $E\mapsto \curl E$ and $H\mapsto\curl H$ may be written in Fourier coordinates as the multiplication operators by the matrices
\[
C(i\xi,\tfrac{\pi }{L_1}n_1,\tfrac{\pi }{L_2}n_2)\qquad 
\text{and}
\qquad
C(i\xi,-\tfrac{\pi }{L_1}n_1,-\tfrac{\pi }{L_2}n_2),
\]
respectively, where the matrix $C$ is defined in \eqref{eq:C(xi)}. As a consequence, in the Fourier domain, $V_\omega^0$ is a multiplication operator represented by the matrix
\[
A_\omega(n,\xi)=\begin{pmatrix}
 i(\omega\epsilon_0+i\sigma_0)I& C(i\xi,-\tfrac{\pi }{L_1}n_1,-\tfrac{\pi }{L_2}n_2) \\
C(i\xi,\tfrac{\pi }{L_1}n_1,\tfrac{\pi }{L_2}n_2) & -i\omega\mu_0 I
\end{pmatrix}.
\]
A further calculation yields
\[
\det\left(A_\omega(n,\xi)\right)=k_\omega(\xi^2+\tfrac{\pi ^2}{L_1^2}n_1^2+\tfrac{\pi ^2}{L_2^2}n_2^2-k_\omega)^2,\qquad k_\omega=\omega\mu_0(\omega\epsilon_0+i\sigma_0).
\]
If $\omega$ is such that $\det\left(A_\omega(n,\xi)\right)\neq 0$ for every $n\in\N^2$ and $\xi\in\R$, then $\omega$ does not belong to the essential spectrum of $V^0_\omega$. On the other hand, suppose that $\omega$ is such that $\det\left(A_\omega(n,\xi)\right) = 0$ for some $n\in\N$ and $\xi\in\R$. If $n_1=n_2=0$, it is easy to see that there are no nonzero elements of $\ker A_\omega(n,\xi)$ satisfying \eqref{eq:constraints}. On the other hand, the vector $(0,\omega\mu_0 L_2,0,\pi i,0,\xi L_2)$ belongs to $\ker A_\omega(0,1,\xi)$ and  satisfies \eqref{eq:constraints} (and similarly if $n_1=1$ and $n_2=0$). As a consequence, we have that
\[
\sigma_{ess}(V^0_\omega)=\{\omega\in\C:k_\omega=0 \text{ or }k_\omega\ge \tfrac{\pi^2}{L^2}\},\qquad L = \max(L_1,L_2).
\]
In the particular case when $\sigma_0=0$, this set takes the simpler form
\[
\sigma_{ess}(V_\omega^0) = (-\infty,-\frac{\pi}{L\sqrt{\epsilon_0\mu_0}}\Big]\cup\{0\}\cup \Big[\frac{\pi}{L\sqrt{\epsilon_0\mu_0}},+\infty).
\]
Note that this set approaches the essential spectrum for the slab as $L\to +\infty$. This is expected: as $L$ increases the cylinder  becomes larger and larger in one direction. 
}
\end{example}

\section{Helmholtz decomposition and related operators}
\label{section:3}
We shall treat both bounded and unbounded Lipschitz domains $\Omega\subseteq\R^3$. The latter are our primary interest, as the bounded case has already been
studied by Lassas \cite{lassas-1998}, albeit under slightly stronger assumptions on the boundary regularity.  However, in the definitions which follow, we deal with both cases. 

The first decomposition result which we require is true without restrictions on the topology of $\Omega$. Although it is standard, we present a proof
since it shows how the homogeneous Sobolev spaces arise in a natural way.
\begin{lem}\label{lemma:mm1} Let $\Omega\subseteq\R^3$ be a Lipschitz domain.
\begin{enumerate}
\item The space $L^2(\Omega;{\mathbb C}^3)$ admits the following orthogonal decompositions:
\begin{subequations}\label{mm:1011}
\begin{align}
 L^2(\Omega;{\mathbb C}^3) &= \nabla \dot{H}_0^1(\Omega)\oplus \Hdivzero, \label{mm:10} \\
 L^2(\Omega;{\mathbb C}^3)  &= \nabla \dot{H}^1(\Omega)\oplus \Hdivzerozero, \label{mm:11} 
\end{align}
\end{subequations}
in which 
\begin{align*}
 \Hdivzero &=  \left\{ u \in L^2(\Omega;{\mathbb C}^3) \; | \; \div u = 0 \right\}, \\
 \Hdivzerozero &=  \left\{ u \in L^2(\Omega;{\mathbb C}^3) \; | \; \div u = 0, \;\; \left. \nu\cdot u\right|_{\partial\Omega}= 0\right\}.
\end{align*}
\item The spaces $\Hocurl$ 
and $\Hcurl$ admit the orthogonal decompositions
\begin{subequations}\label{mm:1213}
\begin{align}
 \Hocurl &= \nabla \dot{H}^1_0(\Omega;{\mathbb C}) \oplus (\Hocurl\cap \Hdivzero), \label{mm:12} \\
\Hcurl &= \nabla \dot{H}^1(\Omega;{\mathbb C}) \oplus (\Hcurl\cap \Hdivzerozero). \label{mm:13}
\end{align}
\end{subequations}
\end{enumerate}
\end{lem}
\begin{proof}
(1) The operator $\nabla\colon \dot{H}^1_0(\Omega)\to L^2(\Omega;{\mathbb C}^3)$ is an isometry, and so $\nabla\dot{H}^1_0(\Omega)$ is closed in $L^2(\Omega;{\mathbb C}^3)$. It remains to prove that $(\nabla\dot{H}^1_0(\Omega))^\perp = \Hdivzero$. Suppose that $\phi\in (\nabla\dot{H}^1_0(\Omega))^\perp$; then $\langle \phi,\nabla v \rangle = 0$ for all $v\in {\mathcal D}(\Omega)$, 
which means that $\langle \div\phi,v \rangle = 0$ for all $v \in {\mathcal D}(\Omega)$. This proves that $\phi \in \Hdivzero$. 
Conversely, if $\phi \in \Hdivzero$ then for any $v\in {\mathcal D}(\Omega)$ we have $0=\langle \div\phi,v\rangle = \langle \phi, \nabla v\rangle$.
Taking the closure in the $\dot{H}^1_0(\Omega)$-topology shows that $\langle \phi,\nabla v \rangle = 0$ for all $v\in \dot{H}^1_0(\Omega)$,
which proves (\ref{mm:10}).

Analogously, $\nabla\dot{H}^1(\Omega)$ is closed in $L^2(\Omega;{\mathbb C}^3)$. To prove (\ref{mm:11}) suppose that $\phi\in (\nabla\dot{H}^1(\Omega))^\perp$; then certainly $\div\phi=0$ since $(\nabla\dot{H}^1(\Omega))^\perp \subseteq (\nabla\dot{H}^1_0(\Omega))^\perp$.
Thus for all $v\in {\mathcal D}(\overline\Omega)$, we have $0=\langle \phi,\nabla v\rangle = \int_{\partial\Omega}(\nu\cdot\phi)\overline{v}\,ds$. This means that $\nu\cdot\phi=0$ on
$\partial \Omega$ and so $\phi\in \Hdivzerozero$. The proof that any $\phi\in\Hdivzerozero$ lies in $(\nabla\dot{H}^1(\Omega))^\perp$ is straightforward.

(2) The decompositions \eqref{mm:1213} follow immediately from \eqref{mm:1011} by taking the appropriate subspaces.
\end{proof}

To decompose the Maxwell pencil we need to decompose the spaces $\Hdivzero$ and $\Hdivzerozero$ further, by using vector potentials in some suitable spaces, which we now introduce.
\begin{defn}\label{def:11}
Let $\Omega\subseteq\R^3$ be a Lipschitz domain. 
\begin{itemize}
\item The space $\dot X_T(\Omega)$ is the closure of $ \Hcurl \cap \Hdivzerozero$ with respect to the seminorm $\| u \| := \| \curl u \|_{L^2(\Omega)} {+\| \div u \|_{L^2(\Omega)}+\|u\cdot\nu \|_{H^{-1/2}(\partial\Omega)}}$.
\item The space $\dot X_N(\Omega)$ is the closure of $\Hocurl \cap \Hdivzero$ with respect to the seminorm $\| u \| := \| \curl u \|_{L^2(\Omega)}{+\| \div u \|_{L^2(\Omega)}+\|u\times\nu \|_{H^{-1/2}(\partial\Omega)}}$.
\item The space $\kt$ is the kernel of the curl operator restricted to $\dot X_T(\Omega)$, namely
\[
\kt =  \{u\in  \dot X_T(\Omega):\curl u=0\}.
\]
\item The space $\kn$ is the kernel of the curl operator restricted to $\dot X_N(\Omega)$, namely
\[
\kn =  \{u\in  \dot X_N(\Omega):\curl u=0\}.
\]
\end{itemize}
\end{defn} 

The spaces  $\kt$ and  $\kn$ are closed in $\dot X_T(\Omega)$ and in  $\dot X_N(\Omega)$, respectively, and so we can consider the quotient spaces
\[
\dot X_T(\Omega)/\kt,\qquad \dot X_N(\Omega)/\kn.
\]
The $\curl$ operator is well-defined and injective on these spaces. To avoid cumbersome notation, we will in the following identify $\curl \psi$ for $\psi\in \dot X_T(\Omega)/\kt$ or $\psi\in \dot X_N(\Omega)/\kn$ with the vector in $L^2(\Omega;{\mathbb C}^3)$ given by $\curl$ acting on any representative of the equivalence class $\psi$. The $\curl$ operator maps these quotient spaces into the space of divergence free fields, with appropriate boundary conditions.
\begin{lem}\label{lem:closed}
Let $\Omega\subseteq\R^3$ be a Lipschitz domain.
\begin{enumerate}
\item The space $\curl(\dot X_T(\Omega)/\kt)$ is contained in $\Hdivzero$.
\item The space $\curl(\dot X_N(\Omega)/\kn)$ is contained in $\Hdivzerozero$.
\end{enumerate}
\end{lem}
\begin{proof}
Part (1) follows immediately from  $\div\circ\curl=0$. Part (2) follows from
 the identities $\div\circ\curl=0$ and $(\curl u)\cdot\nu=\div_{\partial\Omega}(u\times\nu)$ on $\partial\Omega$  \cite[(3.52)]{MONK-2003}.
\end{proof}

We make the following assumption.
\begin{assumption} \label{mm:HD}
The spaces $\kt$ and $\kn$ are finite-dimensional and
\begin{subequations}\label{mm:1516}
\begin{align}
\Hdivzero &= \curl (\dot X_T(\Omega)/\kt) \oplus \kn, \label{mm:15} \\
 \Hdivzerozero &= \curl(\dot X_N(\Omega)/\kn) \oplus \kt.\label{mm:16}
\end{align}
\end{subequations}
\end{assumption}

This assumption is verified in many cases of theoretical and practical interest.
\begin{prop}\label{prop:mm9}
Assumption \ref{mm:HD} is verified in any of the following cases:
\begin{enumerate}
\item $\Omega = \R^3$ {(with $\kt=\kn=\{0\}$)};
\item $\Omega$ is a bounded Lipschitz domain, satisfying Hypothesis 3.3 of \cite{AMROUCHE-BERNARDI-DAUGE-GIRAULT-1998};
\item $\Omega$ is a $C^2$ exterior domain, satisfying  assumptions (1.45) of \cite[Chapter IXA]{dautray-lions-3};
\item $\Omega$ is the half space $\{(x_1,x_2,x_3)\in\R^3:x_3>0\}$ {(with $\kt=\kn=\{0\}$)};
\item $\Omega$ is the slab $\{(x_1,x_2,x_3)\in\R^3:0<x_3<L\}$ for some $L>0$ (with $\kt=\kn=\{0\}$);
\item $\Omega$ is a cylinder $\R\times\Omega'$, where $\Omega'\subseteq\R^2$ is a  simply connected bounded domain of class 
$C^{1,1}$ or piecewise smooth with no re-entrant corners  (with $\kt=\kn=\{0\}$).
\end{enumerate}
\end{prop}
\begin{remark}
We have decided not to provide the details of the assumptions of parts (2) and (3), since they are rather lengthy and are not needed for the rest of the paper. In simple words, these assumptions require $\partial\Omega$ to be a finite union of  connected  surfaces and that there exist a finite number of \emph{cuts} within $\Omega$ which divide it into multiple simply connected domains. The number of cuts is given by $\dim\kt $, and the number of connected components of $\partial\Omega$ by $\dim\kn+1$.  Thus, for simply-connected domains with connected boundaries the decomposition is even
simpler: $\kt$ and $\kn$ are trivial and can be omitted.
\end{remark}
\begin{proof}
(1) The decompositions \eqref{mm:15} and \eqref{mm:16} coincide, and simply follow from the identity $\hat u(\xi) =- \xi\times(\frac{\xi\times \hat u}{|\xi|^2})$, valid for every divergence-free field $u$ (which implies $\xi\cdot\hat u = 0$), where $\hat u$ denotes the Fourier transform of  $u$. Alternatively, this is also a consequence of Proposition~\ref{prop:cylinders} and Lemma~\ref{lem:cylinders}.

(2) This part is proved in \cite{AMROUCHE-BERNARDI-DAUGE-GIRAULT-1998} (see also  \cite[Chapter IXA]{dautray-lions-3} and \cite[Chapter I, $\S$3]{GIRAULT-RAVIART-1986} for the smooth case). The construction of the spaces $\kt$ and $\kn$ is described explicitly.

(3) The decompositions in this case are proved in \cite[Chapter IXA]{dautray-lions-3}.

(4)-(5)-(6) The arguments are standard and explicit, but it is not easy to find precise statements in the literature. We detail the derivation in Appendix~\ref{sec:appendixA}, which contains a general construction for a larger class of cylinders.
\end{proof}

Combining \eqref{mm:1011} and \eqref{mm:1516}, we obtain that the space $L^2(\Omega;{\mathbb C}^3)$ admits the following orthogonal decompositions:
\begin{subequations}\label{eq:Helmholtz-final}
\begin{align}
 L^2(\Omega;{\mathbb C}^3) &= \nabla \dot{H}_0^1(\Omega;\C)\oplus \curl (\dot X_T(\Omega)/\kt) \oplus \kn,\label{eq:Helmholtz-final-1}  \\
 L^2(\Omega;{\mathbb C}^3)  &= \nabla \dot{H}^1(\Omega;\C)\oplus\curl(\dot X_N(\Omega)/\kn) \oplus \kt. \label{eq:Helmholtz-final-2}
\end{align}
\end{subequations}
In view of these decompositions, to every vector field in $ L^2(\Omega;{\mathbb C}^3) $ we can associate the unique vector potentials in $\dot X_T(\Omega)/\kt$ and in $\dot X_N(\Omega)/\kn$.

\begin{lem} \label{T12def} {Let $\Omega\subseteq\R^3$ be a Lipschitz domain satisfying Assumption~\ref{mm:HD}.} There exist bounded operators $T_N:L^2(\Omega;{\mathbb C}^3)\to \dot X_N(\Omega)/\kn$ and $T_T:L^2(\Omega,{\mathbb C}^3)\to \dot X_T(\Omega)/\kt$ such that
 \begin{equation}\label{eq:T1T2}
 \begin{array}{l}
  T_N \curl \Phi = \Phi, \;\; \Phi \in\dot X_N(\Omega)/\kn, \\  T_N\nabla q = 0, \;\; q\in \dot{H}^1(\Omega;{\mathbb C}); \;\;\; T_N f = 0, \;\;\; f \in \kt; \\
  T_T \curl \Phi = \Phi, \;\; \Phi \in \dot X_T(\Omega)/\kt, \\ T_T\nabla q = 0, \;\; q\in \dot{H}^1_0(\Omega;{\mathbb C}); \;\;\; T_T f = 0, \;\;\; f \in \kn.
 \end{array}
 \end{equation}
 \end{lem}
\begin{proof} In view of \eqref{eq:Helmholtz-final-2}, every $F\in L^2(\Omega,{\mathbb C}^3)$ admits a unique decomposition
into three orthogonal vectors,
\[ F = \nabla q + \curl \Phi + f, \]
with $q\in \dot{H}^1(\Omega,{\mathbb C})$, $\Phi\in \dot X_N(\Omega)/\kn$ and $f\in \kt$. We define $T_N$ by $T_NF = \Phi$, so that $T_N\curl\Phi = \Phi$
for all $\Phi\in \dot X_N(\Omega)/\kn$. By  {the closed graph theorem}, $T_N$ is bounded. The definition of $T_T$ follows similarly by using the other Helmholtz decomposition \eqref{eq:Helmholtz-final-1}.
\end{proof}

\section{Proof of the main result}

In a first part, we introduce a series of equivalent reformulations of our problem to obtain a form where the two contributions to the essential spectrum in our main result can easily be separated.

Decomposing $\Honesq$ using \eqref{mm:1213} and \eqref{mm:1516} allows us to transform the Maxwell operator $V_{\omega}$. More precisely, consider the decompositions{
\begin{equation}\label{eq:decomp}E=\nabla q_E+\Psi_E+h_N, \quad  H=\nabla q_H+\Psi_H+h_T,
\end{equation}
where $q_{E}\in \dot{H}_{0}^{1}(\Omega;\C)$,
$q_{H}\in \dot{H}^{1}(\Omega;\C)$, $\Psi_{E}\in \Hocurl\cap \curl (\dot X_T(\Omega)/\kt)$, $\Psi_{H}\in \Hcurl\cap \curl(\dot X_N(\Omega)/\kn)$, $h_T\in\kt$ and $h_N\in\kn$}. We now wish to discard the contribution coming from $\kt$ and $\kn$. To this end, we introduce the space
{
\begin{multline*}
\Hdone_{2}=\nabla \dot{H}^1_0(\Omega)\times\nabla \dot{H}^1(\Omega)\\ \times   \Hocurl\cap  \curl (\dot X_T(\Omega)/\kt)
\times\Hcurl\cap \curl(\dot X_N(\Omega)/\kn)
\end{multline*}
equipped with the canonical product norm 
\begin{equation}\label{eq:normH2}
\norm{(u_1,u_2,\Psi_1,\Psi_2)}_{\Hdone_{2}}^2=\norm{u_1}^2_{L^2(\Omega)}+\norm{u_2}^2_{L^2(\Omega)}+\norm{\Psi_1}^2_\Hcurl+\norm{\Psi_2}^2_\Hcurl.
\end{equation}
}

Define the projection map
\[
W\colon\Honesq\to\Hdone_{2}, 
\qquad W(E,H)={(\nabla q_E,\nabla q_H,\Psi_E,\Psi_H)},
\]
where $E,H$ are given by \eqref{eq:decomp},
and its right inverse $W^{-1}:\Hdone_{2}\to\Honesq$ by
\[
W^{-1}{(\nabla q_E,\nabla q_H,\Psi_E,\Psi_H) = (\nabla q_E+\Psi_E,\nabla q_H+\Psi_H)}.
\]
Since the decompositions in \eqref{mm:1213} and  \eqref{mm:1516} are orthogonal, for any $(E,H)\in\Honesq$ we have
\begin{equation}\label{eq:W}
\norm{(E,H)}_{\Honesq}^2 = \norm{W(E,H)}_{\Hdone_{2}}^2 + \norm{(h_N,h_T)}^2_{L^2(\Omega)^2}.
\end{equation} 
Instead of the operator $V_\omega$, we consider
\[
\tilde{V}_{\omega}=V_{\omega}\circ W^{-1}\colon\Hdone_{2}\to L^2(\Omega;\C^3 )^2.
\]
This does not change the essential spectrum, as
the following lemma shows.
\begin{lem}\label{lem:equivalence2}
The essential spectra of $V_{\omega}$ and of $\tilde{V}_{\omega}$
coincide.
\end{lem}
\begin{proof}
Using that $W^{-1}$ is an isometry we immediately obtain that the essential spectrum of $\tilde V_{\omega}$ is contained in the essential spectrum of $V_{\omega}$. It remains to show the reverse inclusion.

Let $\omega$ belong to the essential spectrum of $V_{\omega}$. By {Remark~\ref{rem:essential} \ref{r2a}}, 
there exists a sequence of functions $u_{n}=(\nabla q_{E,n}+\Psi_{E,n}+h_{N,n}, \nabla q_{H,n}+\Psi_{H,n}+h_{T,n} ) $ in $\Honesq$, $\left\Vert u_{n}\right\Vert _{\Honesq}=1$,
$u_{n}\rightharpoonup0$ in $\Honesq$ such that $\left\Vert V_{\omega}u_{n}\right\Vert _{{L^2}}\to0$.
Then there exists $c>0$ such that $\norm{Wu_n}_{\Hdone_{2}}\geq c$ for all sufficiently large $n$. This follows from the fact that otherwise by \eqref{eq:W} we would have that $P_{NT}u_n:=(h_{N,n},h_{T,n})$ satisfies $\norm{P_{NT}u_n}_{\Hdone_1}\to 1$. However, the range of $P_{NT}$ is the finite dimensional space $\kn\times\kt$. This contradicts that $u_{n}\rightharpoonup0$ in $\Hdone_1$, which implies that $(h_{N,n},h_{T,n})\to 0$ in $\Hdone_1$.

Set
$\tilde{u}_{n}=Wu_{n}/\norm{Wu_{n}}_{\Hdone_{2}}$. Then, $\norm{\tilde{u}_n}_{\Hdone_{2}}=1$ and 
\[
\tilde V_\omega \tilde u_n =\frac{ V_\omega (\nabla q_{E,n}+\Psi_{E,n}, \nabla q_{H,n}+\Psi_{H,n} ) }{\norm{Wu_{n}}_{\Hdone_{2}}}
= \frac{V_\omega u_n- V_\omega (h_{N,n},h_{T,n})}{\norm{Wu_{n}}_{\Hdone_{2}}}\; \longrightarrow\; 0
\]
in $L^2(\Omega;\C^3)^2$.
 Finally, for any $\varphi\in (\Hdone_2)'$ we have $\varphi\circ W\in (\Hdone_1)'$, so 
$$\varphi(\tilde{u}_n)=\frac{(\varphi\circ W)u_n}{\norm{Wu_{n}}_{\Hdone_{2}}} \to 0,$$
and hence $\omega$ is in the essential spectrum of $\tilde{V}_\omega$. 
\end{proof}

{By definition of $\tilde V_\omega$ and \eqref{eq:maxwell_operator}, we obtain
\begin{equation}\label{eq:Vtilde}
\tilde{V}_{\omega}(\nabla q_E,\nabla q_H,\Psi_E,\Psi_H)=
\begin{pmatrix}
\curl\Psi_{H}+i M_{\omega}\nabla q_{E}+ i M_{\omega}\Psi_{E}\\
\curl\Psi_{E}-i\omega M_{\mu}\nabla q_{H}- i\omega M_{\mu}\Psi_{H}
\end{pmatrix}
,
\end{equation}
where $M_{\omega}F=(\omega\epsilon+i\sigma)F$ and $M_{\mu}F=\mu F$.}

In order to simplify this operator even further, we need the following
elementary result.

\begin{lem}
\label{lem:elementary-1}
 Let $P_H$ denote the orthogonal projection onto the space $H$.
\begin{enumerate}
\item The map ${\zeta_1}\colon L^{2}(\Omega;\C^{3})\to \dot{H}^{-1}(\Omega;\C)\times (\dot X_T(\Omega)/\kt)\times \kn$
defined by
\[
F\longmapsto(\div F,T_TF,P_{\kn}F)
\]
is an isomorphism, where $\dot{H}^{-1}(\Omega;\C)$ denotes the dual of $\dot{H}^{1}_0(\Omega;\C)$.
\item The map ${\zeta_2}\colon L^{2}(\Omega;\C^{3})\to \left(\nabla \dot{H}^{1}(\Omega;\C)\right)'\times( \dot X_N(\Omega)/\kn)\times \kt$
given by
\[
F\longmapsto(h(F),T_NF, P_{\kt}F),
\]
where $h\colon L^{2}(\Omega;\C^{3})\to \left(\nabla \dot{H}^{1}(\Omega;\C)\right)'$ is defined by 
$$\langle h(F),\nabla q\rangle:=\int_\Omega F\cdot\nabla q\ dx,$$
is an isomorphism.
\end{enumerate}
\end{lem}

\begin{proof}
(1) Take $(\phi,\Phi,f)\in \dot{H}^{-1}(\Omega;\C)\times( \dot X_T(\Omega)/\kt)\times \kn.$ We need to show that there exists a unique $F\in L^2(\Omega;\C^3)$ such that $\zeta_1(F)=(\phi,\Phi,f)$.
We use the Helmholtz decomposition \eqref{eq:Helmholtz-final-1} and look for $F$ of the form $F=\nabla q+\curl\tilde{\Phi}+f_N$,
with $q\in \dot{H}_{0}^{1}(\Omega;\C)$, $\tilde{\Phi}\in \dot X_T(\Omega)/\kt$ and $f_N\in\kn$. First, since $P_{\kn}F = f_N$, choose $f_N=f$. Now note
that
\[
\div F=\phi\Longleftrightarrow\Delta q=\phi,
\]
which is uniquely solvable for $q\in \dot{H}_{0}^{1}(\Omega;\C)$ {by the Lax-Milgram theorem}. 

Further, 
\[
T_TF=\Phi\Longleftrightarrow\tilde{\Phi}=\Phi,
\]
which is clearly uniquely solvable for $\tilde{\Phi}\in \dot X_T(\Omega)/\kt$.
This shows that $\zeta_1(\nabla q+\curl\tilde{\Phi}+f_N)=(\phi,\Phi,f)$, as desired.

(2) The map $\zeta_2$ is well-defined since $\nabla\dot{H}^1(\Omega)\subseteq L^2(\Omega;\C^3)$.
We now show that $\zeta_2$ is an isomorphism.
Take $$(\varphi,\Phi,f)\in \left(\nabla\dot{H}^{1}(\Omega;\C)\right)'\times(\dot X_N(\Omega)/\kn)\times \kt.$$
We use the Helmholtz decomposition \eqref{eq:Helmholtz-final-2} and look for $F$ of the form $F=\nabla p+\curl\tilde\Phi+f_T$,
with $p\in \dot{H}^{1}(\Omega;\C)$, $\tilde\Phi\in\dot X_N(\Omega)/\kn$ and $f_T\in\kt$. Then 
$T_NF=\tilde\Phi$
and $P_{\kt}F=f_T$, and so $\tilde\Phi$ and $f_T$ are uniquely determined by $\tilde\Phi=\Phi$ and $f_T=f$. 

It remains to show that $p$ can be chosen so that $\nabla p+\curl\Phi+f=\varphi$ or $\nabla p=\varphi-\curl\Phi-f$ in $\left(\nabla\dot{H}^{1}(\Omega;\C)\right)'$.
Thus we need to find $p$ such that 
$$ \int_\Omega \nabla p\cdot\nabla q\ dx = \int_\Omega (\varphi-\curl \Phi-f) \cdot\nabla q\ dx,\qquad q\in \dot{H}^{1}(\Omega;\C).$$
Using that $ L^2(\Omega;\C^3)\subseteq\left(\nabla\dot{H}^1(\Omega)\right)'$, this is uniquely solvable for $p$ using the Lax-Milgram theorem.

This shows that $\zeta_2(\nabla p+\curl\Phi+f)=(\varphi,\Phi,f)$, as desired.
\end{proof}

{Now, define $\zeta=(\zeta_1,\zeta_2)$ and $\tilde\zeta(F)=(\div F,h(F),T_N F,T_T F)$, i.e. $\tilde{\zeta}$ contains the parts of $\zeta$ not in $\kn\oplus\kt$.   Let 
\[
\Hdone_3= \dot{H}^{-1}(\Omega;\C)\times\left(\nabla \dot{H}^{1}(\Omega;\C)\right)'\times(\dot X_N(\Omega)/\kn)\times(\dot X_T(\Omega)/\kt).
\]
 }
Set $$\tilde{\tilde{V}}_\omega=\tilde{\zeta}\circ\tilde{V}_\omega:\Hdone_2\to\Hdone_3.$$

\begin{lem}\label{lem:equivalence3}
The essential spectra of $\tilde{V}_\omega$ and of $\tilde{\tilde{V}}_\omega$
coincide.
\end{lem}

\begin{proof}
This follows from the fact that by Lemma \ref{lem:elementary-1}, $\zeta$ is a bijective continuous linear map, so both $\zeta$ and $\zeta^{-1}$ are continuous, and that $\kn\oplus\kt$ is finite dimensional.
\end{proof}

Now, 
recalling that $\Psi_H\in\dot X_T(\Omega)$ and $\Psi_E\in\dot X_N(\Omega)$, by \eqref{eq:T1T2} and \eqref{eq:Vtilde} we have that 
{
\begin{equation}\label{eq:Vtt}
\begin{split}
\tilde{\tilde{V}}_\omega(\nabla q_E,\nabla q_H,\Psi_E,\Psi_H) 
&=\tilde{\zeta}\left(\begin{array}{c} \curl\Psi_{H}+i M_{\omega}\nabla q_{E}+ i M_{\omega}\Psi_{E}\\
\curl\Psi_{E}-i\omega M_{\mu}\nabla q_{H}- i\omega M_{\mu}\Psi_{H} \end{array}\right) \\
&=\left(\begin{array}{c} i\div(M_{\omega}\nabla q_{E})+i\div(M_{\omega}\Psi_{E})\\
-i\omega h(M_{\mu}\nabla q_{H})-i\omega h(M_{\mu}\Psi_{H}) \\
{[\Psi_{E}]}-i\omega T_NM_{\mu}\nabla q_{H}-i\omega T_NM_{\mu}\Psi_{H}\\
{[\Psi_{H}]}+i T_TM_{\omega}\nabla q_{E}+i T_TM_{\omega}\Psi_{E}
\end{array}\right),
\end{split}
\end{equation}
}
in which $[ \cdot ]$ denotes the equivalence class in the appropriate quotient space.

In order to compute the essential spectrum of $\tilde{\tilde{V}}_\omega$ we now decompose the coefficients in the Maxwell system.
As a consequence of our hypotheses (\ref{eq:vai},\ref{eq:vaib}), whether $\Omega$ be bounded or unbounded, for each $\delta>0$ the Maxwell 
coefficients admit a decomposition 
\begin{equation} \mu = \mu_0 + \mu_c + \mu_\delta, \;\;\; \epsilon = \epsilon_0 + \epsilon_c + \epsilon_\delta, \;\;\;
 \sigma = \sigma_0 + \sigma_c + \sigma_\delta, \label{eq:mm4.14} \end{equation}
 in which the terms $\mu_0$, $\epsilon_0$ and $\sigma_0$ are constant and do not depend on $\delta$, the terms $\mu_c$, $\epsilon_c$ and 
 $\sigma_c$ are compactly supported, and the terms $\mu_\delta$, $\epsilon_\delta$, $\sigma_\delta$ are essentially bounded, with
 \begin{equation} m_\delta:=\max(\| \mu_\delta \|_{L^\infty(\Omega)}, \| \epsilon_\delta\|_{L^\infty(\Omega)}, \| \sigma_\delta \|_{L^\infty(\Omega)}) < \delta, \label{eq:mm4.1} \end{equation}
where the norms are defined by $\norm{a}_{L^\infty(\Omega)}:=\esssup_{x\in\Omega}\norm{a(x)}_2$ for $a\in L^\infty(\Omega;\R^{3\times 3})$, where $\norm{A}_2$ denotes the induced norm $\sup_{v\in\R^3\setminus\{0\}}\frac{|Av|}{|v|}$ for $A\in\R^{3\times 3}$.

In the expression for $\tilde{\tilde{V}}_\omega$ appearing in (\ref{eq:Vtt}) the Maxwell coefficients appear linearly in the multiplication operators
$M_\mu$ (multiplication by $\mu$) and $M_\omega$ (multiplication by $\omega \epsilon + i \sigma$). The decomposition (\ref{eq:mm4.14}) 
of the coefficients is partially reflected in the following decomposition of $\tilde{\tilde{V}}_\omega$:
\begin{equation*}
 \tilde{\tilde{V}}_\omega = \tilde{\tilde{V}}_{\omega,0} + \tilde{\tilde{V}}_{\omega,c} + \tilde{\tilde{V}}_{\omega,\delta}, \label{eq:mm4.3} 
\end{equation*}
in which {
\begin{equation}\label{eq:Vtt0}
\tilde{\tilde{V}}_{\omega,0}
\begin{pmatrix}
\nabla q_E \\ \nabla q_H\\ \Psi_E \\ \Psi_H 
 \end{pmatrix}
=
\begin{pmatrix}
 i\div((\omega\epsilon+i\sigma)\nabla q_{E})\\
-i\omega h(\mu\nabla q_{H}) \\
 - i\omega T_N \left((\mu_0+\mu_c) \nabla q_{H}\right)+{[\Psi_E]}-i\omega T_N(\mu_0 \Psi_{H}) \\
 i T_T\left((\omega(\epsilon_0+\epsilon_c)+ i (\sigma_0+\sigma_c))\nabla q_{E}\right) + i T_T ((\omega\epsilon_0+i\sigma_0)\Psi_E) +{[\Psi_H]}   
 \end{pmatrix},
\end{equation}
\begin{equation*}\label{eq:Vttc}
\tilde{\tilde{V}}_{\omega,c}
\begin{pmatrix}
\nabla q_E \\ \nabla q_H\\ \Psi_E \\ \Psi_H 
 \end{pmatrix}
=
\begin{pmatrix}
 i\div((\omega(\epsilon_0+\epsilon_c)+i(\sigma_0+\sigma_c))\Psi_{E})\\
-i\omega  h((\mu_0+\mu_c)\Psi_{H}) \\
-i\omega T_N(\mu_c \Psi_{H}) \\
i T_T\left((\omega\epsilon_c + i \sigma_c)\Psi_{E}\right)
 \end{pmatrix}
\end{equation*}
and
\begin{equation*}\label{eq:Vttd}
\tilde{\tilde{V}}_{\omega,\delta}
\begin{pmatrix}
\nabla q_E \\ \nabla q_H\\ \Psi_E \\ \Psi_H 
 \end{pmatrix}
=
\begin{pmatrix}
 i\div((\omega\epsilon_\delta+i\sigma_\delta)\Psi_{E})\\
-i\omega h(\mu_\delta \Psi_{H}) \\
-i\omega T_N( \mu_\delta \nabla q_{H})-i\omega T_N(\mu_\delta \Psi_{H})\\
i T_T((\omega\epsilon_\delta + i \sigma_\delta)\nabla q_{E})+i T_T((\omega\epsilon_\delta + i \sigma_\delta)\Psi_{E})
 \end{pmatrix}
.
\end{equation*}
The  operator $\tilde{\tilde{V}}_{\omega,c}$  is compact and the  operator $\tilde{\tilde{V}}_{\omega,\delta}$ is $O(\delta)$-small in a suitable norm, as we show in the following two lemmata.
\begin{lem}\label{prop:mm4.1}
The operator $\tilde{\tilde V}_{\omega,c}:\Hdone_2\to\Hdone_3$ is compact.
\end{lem}
\begin{proof}
By a direct calculation it is easy to see that $\div((\omega\epsilon_0+i\sigma_0)\Psi_E)=0$ and $h(\mu_0\Psi_H)=0$, using that $\epsilon_0$, $\sigma_0$ and $\mu_0$ are scalar. Since the operators
\begin{equation}\label{eq:4operators}
\begin{aligned}
 \div\colon L^2&(\Omega;\C^3)\to \dot H^{-1}(\Omega;\C),&
T_T\colon L^2&(\Omega;\C^3)\to \dot X_T(\Omega)/\kt,\\
 h \colon L^2&(\Omega;\C^3)\to \bigl(\nabla\dot H^{1}(\Omega;\C)\bigr)',&
T_N\colon L^2&(\Omega;\C^3)\to \dot X_N(\Omega)/\kn,
\end{aligned}
\end{equation}
are bounded, it is enough to show that the operators
\begin{align*}
&F_T\colon \Hocurl\cap  \curl (\dot X_T(\Omega)/\kt)\to L^2(\Omega;\C^3),\qquad \Psi_E\mapsto (\omega\epsilon_c+i\sigma_c)\Psi_E,\\
&F_N\colon \Hcurl\cap  \curl (\dot X_N(\Omega)/\kn)\to L^2(\Omega;\C^3),\qquad \Psi_H\mapsto\mu_c\Psi_H,
\end{align*}
are compact. We now prove that $F_T$ is compact, the other proof is completely analogous. Let $R>0$ be big enough so that $K:=\supp(\omega\epsilon_c+i\sigma_c)\subseteq B(0,R)\cap\overline\Omega$ and $\chi\in C^\infty(\Omega)$ be a cutoff function such that $\chi\equiv 1$ in $K$ and $\supp \chi\subseteq B(0,R)\cap\overline\Omega$. Setting $\Omega_R=B(0,R)\cap\Omega$, the operator $F_T$ may be expressed via the following compositions
\[
\begin{array}{cc}
\Hocurl\cap  \curl (\dot X_T(\Omega)/\kt)  &  \Psi_E\\
\downarrow &\downmapsto\\
 \mathcal H_0(\curl,\Omega_R)\cap \mathcal H(\div,\Omega_R) & (\chi\Psi_E)|_{\Omega_R} \\
\hookdownarrow &\downmapsto\\
L^2(\Omega_R;\C^3) & (\chi\Psi_E)|_{\Omega_R}\\
\downarrow &\downmapsto\\
L^2(\Omega_R;\C^3) &( (\omega\epsilon_c+i\sigma_c)\Psi_E)|_{\Omega_R} \\
\hookdownarrow&\downmapsto\\
L^2(\Omega;\C^3) &(\omega\epsilon_c+i\sigma_c)\Psi_E \\
\end{array},
\]
where the third operator is  the multiplication by $ \omega\epsilon_c+i\sigma_c$ {and the fourth operator is simply the extension by zero}.
Therefore, since the embedding $\mathcal H_0(\curl,\Omega_R)\cap \mathcal H(\div,\Omega_R) \hookrightarrow L^2(\Omega_R;\C^3)$ is compact  \cite{weber1980} (see also \cite[Theorem 2.8]{AMROUCHE-BERNARDI-DAUGE-GIRAULT-1998}), the operator $F_T$ is compact.
\end{proof}
}

\begin{lem}\label{lemma:mm4.4} There exists a constant $C>0$ depending only on $\Omega$ and on the coefficients $\mu$, $\epsilon$ and $\sigma$, such that
for each $\delta>0$ {we have}
\[ \| \tilde{\tilde{V}}_{\omega,\delta} \|_{\Hdone_2\to\Hdone_3} \leq C (1+|\omega|)\delta. \]
\end{lem}
\begin{proof}
{
Note that by \eqref{eq:normH2} we have
\[
\|(\nabla q_E,\nabla q_H, \Psi_E,\Psi_H)\|_{L^2(\Omega;\C^3)^4}
\le \|(\nabla q_E,\nabla q_H, \Psi_E,\Psi_H)\|_{\Hdone_2}.
\]
Thus, since the four operators in \eqref{eq:4operators} are bounded, there exists a constant $C>0$ depending only on $\Omega$ and on the coefficients $\mu$, $\epsilon$ and $\sigma$, such that
\[
\begin{split}
\| \tilde{\tilde V}_{\omega,\delta}(\nabla q_E,\nabla q_H, \Psi_E,\Psi_H)\|_{\Hdone_3}
& \le C(1+|\omega|) m_\delta \|(\nabla q_E,\nabla q_H, \Psi_E,\Psi_H)\|_{L^2(\Omega;\C^3)^4}\\
& \le C(1+|\omega|) \delta \|(\nabla q_E,\nabla q_H, \Psi_E,\Psi_H)\|_{\Hdone_2},
\end{split}
\]
where the second inequality follows from \eqref{eq:mm4.1}. This concludes the proof.
}
\end{proof}

It is helpful to recall that $\tilde{\tilde{V}}_{\omega}\colon \Hdone_{2}\to \Hdone_3 $, where
\begin{multline*}
\Hdone_{2}=\nabla \dot{H}^1_0(\Omega)\times\nabla \dot{H}^1(\Omega)\\ \times   \Hocurl\cap  \curl (\dot X_T(\Omega)/\kt)
\times\Hcurl\cap \curl(\dot X_N(\Omega)/\kn)
\end{multline*}
and
\[
\Hdone_3= \dot{H}^{-1}(\Omega;\C)\times\left(\nabla \dot{H}^{1}(\Omega;\C)\right)'\times(\dot X_N(\Omega)/\kn)\times(\dot X_T(\Omega)/\kt).
\]

\begin{prop}\label{prop:mm4.3} The \textcolor{black}{essential} spectrum of  {$\tilde{\tilde{V}}_{\omega}$} is the union of the \textcolor{black}{essential} spectra of the two block operator pencils {
\begin{align*}
\mathcal A_\omega&\colon   \nabla \dot{H}^1_0(\Omega)\times\nabla \dot{H}^1(\Omega)\to 
\dot{H}^{-1}(\Omega;\C)\times\left(\nabla \dot{H}^{1}(\Omega;\C)\right)' \\
\mathcal D_\omega& \colon  \Hocurl\cap  \curl (\dot X_T(\Omega)/\kt)
\times\Hcurl\cap \curl(\dot X_N(\Omega)/\kn) \\
&\to (\dot X_N(\Omega)/\kn)\times(\dot X_T(\Omega)/\kt)
\end{align*}
}
determined
by the expressions
\begin{equation}\label{Adef}
 \mathcal A_\omega 
=
\begin{pmatrix}  i\div((\omega\epsilon+i\sigma) \,\cdot ) & 0  \\ 0 &- i\omega h(\mu \, \cdot)  \end{pmatrix}
,\quad \mathcal D_\omega 
=
\begin{pmatrix} I_N &- i\omega\mu_0 T_N \\ i(\omega\epsilon_0+i\sigma_0) T_T & I_T \end{pmatrix}.
\end{equation}
Here the operators $I_N$ and $I_T$ are the canonical mappings from $\dot X_N(\Omega)$ and $\dot X_T(\Omega)$ to the quotient spaces $\dot X_N(\Omega)/\kn$ and
$\dot{X}_T(\Omega)/\kt$ respectively.
\end{prop}

\begin{proof}
By inspection of (\ref{eq:Vtt0}), the operator pencil $\tilde{\tilde V}_{\omega,0}$ may be written as the
block lower triangular operator matrix pencil
\begin{equation}\label{eq:ACD}
\tilde{\tilde V}_{\omega,0}= 
\begin{pmatrix}
 {\mathcal A_\omega} & 0 \\  {\mathcal C_\omega} & {\mathcal D_\omega}
\end{pmatrix}
,
\end{equation}
in which $\mathcal A_\omega$ and $\mathcal D_\omega$ are as in equation \eqref{Adef} and  the off-diagonal component ${\mathcal C_\omega}\colon \nabla \dot{H}^1_0(\Omega)\times\nabla \dot{H}^1(\Omega)\to  (\dot X_N(\Omega)/\kn)\times(\dot X_T(\Omega)/\kt)$ is given by 
\[ 
{\mathcal C_\omega} =
\begin{pmatrix}
0 &- i\omega T_N((\mu_0+\mu_c)\cdot)  \\
    iT_T ((\omega(\epsilon_0+\epsilon_c)+i(\sigma_0+\sigma_c))\cdot) & 0
\end{pmatrix}
.
 \]
{Take $\omega\notin\sigma_{ess}(\mathcal A_\omega)\cup \sigma_{ess}(\mathcal D_\omega)$. By Remark~\ref{rem:essential}\ref{r2c}, there exist compact operators $K_1$ and $K_2$ such that $\mathcal A_\omega+K_1$ and $\mathcal D_\omega+K_2$ are invertible. Thus, $\tilde{\tilde V}_{\omega,0}+\left(\begin{smallmatrix}
K_1 &0  \\
    0 & K_2
\end{smallmatrix}\right)$ is invertible, and its inverse is given by 
\[
\left(\tilde{\tilde V}_{\omega,0}+\begin{pmatrix}
K_1 &0  \\
    0 & K_2
\end{pmatrix}\right)^{-1}=
\begin{pmatrix}
(\mathcal A_\omega+K_1)^{-1} &0  \\
    -(\mathcal D_\omega+K_2)^{-1}\mathcal C_\omega (\mathcal A_\omega+K_1)^{-1}  & (\mathcal D_\omega+K_2)^{-1}
\end{pmatrix}.
\]
Since $\mathcal A_\omega$ and $\mathcal D_\omega$ are independent of $\delta$ and the operator norm of $\mathcal C_\omega$ may be bounded independently of $\delta$, this inverse is bounded in norm by a constant independent of $\delta$. As a consequence, by Lemma~\ref{lemma:mm4.4} there exists $\delta>0$ such that $\tilde{\tilde V}_{\omega,0}+\left(\begin{smallmatrix}
K_1 &0  \\
    0 & K_2
\end{smallmatrix}\right) + \tilde{\tilde V}_{\omega,\delta}$ is invertible. Using again Remark~\ref{rem:essential}\ref{r2c}, we obtain that $\omega\notin\sigma_{ess}(\tilde{\tilde V}_{\omega,0}+\tilde{\tilde V}_{\omega,\delta})$. Finally, Lemma~\ref{prop:mm4.1} yields $\omega\notin\sigma_{ess}(\tilde{\tilde V}_{\omega,0}+\tilde{\tilde V}_{\omega,\delta}+\tilde{\tilde V}_{\omega,c})=\sigma_{ess}(\tilde{\tilde V}_{\omega})$, since the essential spectrum  is invariant under compact perturbations.

Now take $\omega\notin\sigma_{ess}(\tilde{\tilde V}_{\omega})$. By Remark~\ref{rem:essential}\ref{r2c} there exists a compact operator $K$ such that $\tilde{\tilde V}_{\omega} + K$ is invertible. By Lemma~\ref{lemma:mm4.4} there exists $\delta>0$ such that $\tilde{\tilde V}_{\omega,0}+\tilde{\tilde V}_{\omega,c}+K=\tilde{\tilde V}_{\omega} + K-\tilde{\tilde V}_{\omega,\delta}$ is invertible, whence $\omega\notin\sigma_{ess}(\tilde{\tilde V}_{\omega,0}+\tilde{\tilde V}_{\omega,c})=\sigma_{ess}(\tilde{\tilde V}_{\omega,0})$. Assume by contradiction that 
$\omega\in\sigma_{ess}(\mathcal A_\omega)\cup \sigma_{ess}(\mathcal D_\omega)$. If $\omega\in \sigma_{ess}(\mathcal D_\omega)$, by Remark~\ref{rem:essential}\ref{r2a} there exists a singular sequence $(v_n)_n$ for $\mathcal D_\omega$, and so  $(0,v_n)_n$ is a singular sequence for $\tilde{\tilde V}_{\omega,0}$ by \eqref{eq:ACD}, which implies that $\omega\in\sigma_{ess}(\tilde{\tilde V}_{\omega,0})$, a contradiction. Otherwise, if $\omega\in \sigma_{ess}(\mathcal A_\omega)\setminus\sigma_{ess}(\mathcal D_\omega)$, there exists a compact operator $K'$ such that $\mathcal D_\omega+K'$ is invertible and a singular sequence $(u_n)_n$ for $\mathcal A_\omega$. A direct calculation then shows that $(u_n,-(\mathcal D_\omega+K')^{-1}\mathcal C_\omega u_n)_n$ is a singular sequence for $\tilde{\tilde V}_{\omega,0}$,  which again implies that $\omega\in\sigma_{ess}(\tilde{\tilde V}_{\omega,0})$. 
} \end{proof}

\begin{remark}
The text \cite{MR2463978} contains many interesting results on essential spectra of block-operator matrices and pencils;  Theorem 2.4.1 is very close to what
we would need, but our pencil ${\tilde{\tilde V}}_{\omega,0}$ is lower triangular rather than diagonally dominant.
\end{remark}

{We are now ready to prove our main result.}

\begin{proof}[Proof of Theorem \ref{thm:adv}]
We commence the proof by observing the following  {identity}:
\begin{equation}\label{chain} \sigma_{ess}(V_\omega) = \sigma_{ess}({\mathcal A}_\omega)\cup \sigma_{ess}({\mathcal D}_\omega). \end{equation}
This is an immediate consequence of Lemmas~\ref{lem:equivalence2} and \ref{lem:equivalence3} and of Proposition \ref{prop:mm4.3}. 
We now consider $\sigma_{ess}({\mathcal A}_\omega)$ and $ \sigma_{ess}({\mathcal D}_\omega)$ in more detail.

The essential spectrum of ${\mathcal A}_\omega$ consists of the point $\{ 0 \}$, arising from the $(2,2)$ diagonal entry of ${\mathcal A}_\omega$, which has $\omega=0$ as
an eigenvalue of infinite multiplicity and is otherwise invertible;  and of the essential spectrum of the pencil in the $(1,1)$ entry, which is as stated
in the theorem, {namely
\begin{equation}\label{eq:Aomega} 
\sigma_{ess}({\mathcal A}_\omega) = \{0\}\cup\sigma_{ess} (\div((\omega\epsilon+i\sigma)\nabla\,\cdot\,)).
\end{equation}
}

In order to deal with the essential spectrum of ${\mathcal D}_\omega$ we observe that if we replace $V_\omega$ by a new pencil $V_\omega^0$  in which the coefficients have the constant
values $\epsilon_0$, $\mu_0$ and $\sigma_0$, then ${\mathcal D}_\omega$ will be unchanged while ${\mathcal A}_\omega$ will be replaced by a pencil ${\mathcal A}_{\omega,0}$ in which all the coefficients are constant.  For the constant coefficient pencil ${\mathcal A}_{\omega,0}$ we see that $0$ lies in the essential spectrum as we reasoned before,
while the $(1,1)$ term is invertible and Fredholm precisely when $\omega\epsilon_0+i\sigma_0\neq 0$, by the Babu\v{s}ka-Lax-Milgram theorem; hence
$ \sigma_{ess}({\mathcal A}_{\omega,0}) =  \{0,-i\sigma_0/\epsilon_0\}.$
Using (\ref{chain}) for the constant coefficient pencil, we now have
\begin{equation}\label{cheaptrick} 
 \sigma_{ess}(V_\omega^0) = \{0,-i\sigma_0/\epsilon_0\}\cup \sigma_{ess}({\mathcal D}_\omega). 
\end{equation}

We now prove that the essential spectrum of ${\mathcal A}_\omega$ already contains $\{ 0, -i\sigma_0/\epsilon_0\}$. The $(2,2)$ component
has $0$ as an eigenvalue of infinite multiplicity.  If $\Omega$ is bounded, we have $\sigma_0=0$ and so the claim is proven. Otherwise,   for the point $-i\sigma_0/\epsilon_0$ we observe that by the hypothesis (\ref{eq:vai}), given $n>0$ there exists $R_n>0$
such that if $\omega_0 := -i\sigma_0/\epsilon_0$ then
\[ \sup_{| x | \geq R_n}  \| \omega_0 \epsilon(x) + i \sigma(x) \|_{2} < \frac{1}{n}. \]
Choosing any function $\phi_n\in C_0^\infty(\Omega)$ with support in $\{ x\in\Omega : | x| > R_n \}$, with $\| \nabla \phi_n \|_{L^2(\Omega)} = 1$, we see that
\[ \| \div((\omega_0\epsilon+i\sigma)\nabla \phi_n) \|_{\dot{H}^{-1}(\Omega)} \leq \frac{1}{n}. \]
Since the supports of the sequence $(\nabla \phi_n)_{n\in\mathbb N}$ move off to infinity, the sequence converges weakly to zero; it  is therefore a singular sequence in $\nabla \dot{H}^1_0(\Omega)$
for the $(1,1)$ element of ${\mathcal A}_{\omega_0}$. Thus $\omega_0$ lies in the essential spectrum of ${\mathcal A}_\omega$.
Combining the observations {\eqref{chain}, \eqref{eq:Aomega} and} \eqref{cheaptrick} with the fact that $\sigma_{ess}({\mathcal A}_\omega) \supseteq \{ 0, -i\sigma_0/\epsilon_0\}$ completes the proof.
\end{proof}

We conclude this section with a more explicit description of the essential spectrum of the divergence form operator $\div((\omega\epsilon+i\sigma) \nabla\,\cdot\, )$ in the case of continuous coefficients.

\begin{prop}\label{prop:mm164}
When the coefficients $\epsilon$ and $\sigma$ are continuous in $\overline\Omega$, the essential spectrum of $\div((\omega\epsilon+i\sigma) \nabla\,\cdot\, )$, acting 
from $\dot{H}^1_0(\Omega;\C)$ to $\dot{H}^{-1}(\Omega;\C)$, consists of the closure of the set of all  $\omega=i\nu$, $\nu\in {\mathbb R}$, such that $\nu\varepsilon+\sigma$ is indefinite
at some point in $\Omega$. Equivalently, when $\Omega$ is bounded, it is the set of $\omega=i\nu$, $\nu\in\mathbb R$, such that $\nu\varepsilon+\sigma$
is indefinite at some point in $\overline\Omega$.
\end{prop}
\begin{proof}
If $\Re(\omega)\neq 0$ then the real part of $\omega\varepsilon+i\sigma$ is definite, and the result follows by the Lax-Milgram theorem. If $\omega=i\nu$ is purely imaginary, this reasoning
still works if $\nu\varepsilon+\sigma$ is {uniformly} definite {in $\Omega$}. It remains only to show that if $\nu\varepsilon+\sigma$ is indefinite at some point $x_0\in\Omega$, then $0$ lies in the essential spectrum of 
$\div((\omega\epsilon+i\sigma) \nabla\,\cdot\, )$.

We prove the result by constructing a Weyl singular sequence. Define $a:=\nu\varepsilon+\sigma$ and $a_0:= a(x_0)$. Let $\chi:[0,\infty)\mapsto [0,1]$ be a smooth cutoff
function such that $\chi(t)=1$ for $0\leq t \leq 1$ and $\chi(t)=0$ for all $t\geq 2$. Let $\theta\in {\mathbb R}^3$ be a unit vector chosen such that
$\theta^Ta(x_0)\theta=0$. For each sufficiently small $\delta>0$ and large $r>0$ let 
\begin{equation}\label{mm:urd} \chi_\delta(x) := \frac{1}{\delta^{3/2}}\chi\left(\frac{|x-x_0|}{\delta}\right), \qquad u_{r,\delta}(x) := \chi_\delta(x)r^{-1}\exp(ir\theta\cdot x). \end{equation}

A direct calculation shows that $\nabla u_{r,\delta}$ in sup-norm is $O(r^{-1}\delta^{-5/2}) + O(\delta^{-3/2})$. We suppose that $r\delta^\frac{5}{2} \gg 1$, so that the $\delta^{-3/2}$
term dominates; {we have $\| \nabla u_{r,\delta} \|_{L^\infty(B_{2\delta}(x_0))} = O(\delta^{-3/2})$} and
$\| u_{r,\delta} \|_{H^1_0(\Omega)} \ge c$ for some $c>0$ independent of $r$ and $\delta$. If $v$ is any smooth test function then 
\[ 
|\langle \nabla u_{r,\delta}, \nabla v \rangle|  \leq C\delta^3\| \nabla u_{r,\delta} \|_{L^\infty(B_{2\delta}(x_0))}\| \nabla v \|_{L^\infty(B_{2\delta}(x_0))} 
\leq C\| \nabla v \|_{\infty}  \delta^{3} \delta^{-3/2}= O(\delta^{3/2}),
 \]
so that the $u_{r,\delta}$ tend to zero weakly in $H^1_0(\Omega)$ as $r\nearrow+\infty$ and $\delta\searrow 0$, {with $r \ge \delta^{-\frac{5}{2}}$}.

To complete the proof that $0$ lies in the essential spectrum of our operator we show that $\| \div(a\,\nabla u_{r,\delta})\|_{H^{-1}(\Omega)}$ can be made arbitrarily small.  It is easy to see that 
\begin{equation}\label{mm:ws1} \| \div(a\,\nabla u_{r,\delta})\|_{H^{-1}(\Omega)} \leq \| a - a_0 \|_{L^\infty(B_{2\delta}(x_0))} + \sup_{v\in H^1_0(\Omega)\setminus\{0\}}
\hspace{-4mm}
 \frac{\left| \langle a_0\,\nabla u_{r,\delta},\nabla v \rangle\right|}{\| v \|_{H^1_0(\Omega)}}. \end{equation}
We compute $\nabla u_{r,\delta} $ by direct differentiation of eqn.\ (\ref{mm:urd}) and deduce that for each $v\in H^1_0(\Omega)$,
\[
\begin{split} 
\left| \langle a_0\,\nabla u_{r,\delta} ,\nabla v \rangle\right| & \leq  \left|\left\langle a_0 \frac{1}{r\delta^\frac{5}{2}}\frac{x-x_0}{|x-x_0|}\chi'\left(\frac{|x-x_0|}{\delta}\right)\exp(ir\theta\cdot x),\nabla v \right\rangle \right| \\
&   \qquad + \left|\left\langle \chi_\delta \, a_0\nabla\left(r^{-1}\exp(ir\theta\cdot x)\right),\nabla v \right\rangle\right| \\
& \leq  {C} \frac{|a_0|}{r\delta}\| v \|_{H^1_0(\Omega)} + \left|\left\langle \div(\chi_\delta \, a_0 \, \nabla\left(r^{-1}\exp(ir\theta\cdot x)\right) ,v\right\rangle\right| \\
 & = {C} \frac{|a_0|}{r\delta}\| v \|_{H^1_0(\Omega)} + \left|\left\langle (\nabla \chi_\delta)\cdot a_0 \, \nabla\left(r^{-1}\exp(ir\theta\cdot x)\right),v\right\rangle\right|; 
 \end{split}
 \]
 in the last step we have used the fact that $\div\left(a_0\,\nabla\left(r^{-1}\exp(ir\theta\cdot x)\right)\right)=0$, which follows immediately from  $\theta^Ta_0\theta = 0$. Integration by parts yields
 \[  \left| \langle a_0\,\nabla u_{r,\delta} ,\nabla v \rangle\right|  \leq {C} \frac{|a_0|}{r\delta}\| v \|_{H^1_0(\Omega)} + \left|\left\langle  r^{-1}\exp(ir\theta\cdot x), \div(v a_0^T \nabla \chi_\delta) \right\rangle\right|. \]
We estimate the final inner product by observing that $\chi_\delta $ is $O(\delta^{-3/2})$, its gradient is $O(\delta^{-5/2})$ and its second derivatives $O(\delta^{-7/2})$, while its support is a ball 
whose volume is $O(\delta^3)$: thus
\[  \left| \langle a_0\,\nabla u_{r,\delta},\nabla v\rangle\right|  \leq  C\frac{|a_0|}{r}\left\{ \frac{1}{\delta}+\frac{1}{\delta^2}\right\}\| v \|_{H^1_0(\Omega)}, \]
for some constant $C>0$. Substituting this back into (\ref{mm:ws1}) we obtain
\[ \| \div(a\,\nabla u_{r,\delta})\|_{H^{-1}(\Omega)} \leq \| a - a_0 \|_{L^\infty(B_{2\delta}(x_0))} + C\frac{|a_0|}{r}\left\{\frac{1}{\delta}+\frac{1}{\delta^2}\right\}. \]
Letting $r\nearrow \infty$ and then letting $\delta\searrow 0$ we obtain the required result.
\end{proof}

\renewcommand{\norm}[1]{||#1||}

\appendix
\section{The Helmholtz decomposition for cylinders}\label{sec:appendixA}

This appendix is devoted to the study of the decompositions \eqref{mm:15} and \eqref{mm:16} for a large class of cylinders of the form $\Omega=\R\times \Omega'$, with $\Omega'\subseteq\R^2$. We will then show that this class includes the full space, the half-space, the slab, and the cylinders with bounded sections as in Proposition~\ref{prop:mm9} part(6), thereby providing a proof to the corresponding parts of Proposition~\ref{prop:mm9}.

 We denote coordinates 
in $\Omega$ by $(x_1,x')$ where $x'=(x_2,x_3)\in\Omega'$, with similar conventions for components of vectors  and operators, such as gradient and Laplacian. For simplicity of notation, we shall write $a\lesssim b$ to mean $a\le C b$ for some positive constant $C$ depending only on $\Omega'$. We assume that
the cross-section $\Omega'$ satisfies the following additional hypothesis.
\begin{assumption}  \label{ch1} Let $g,h\in L^2(\Omega')$. If $\psi' \in\mathcal{D}'(\Omega')$ satisfies 
\begin{equation}\label{eq:hp1}
\left\{ 
 \begin{array}{ll} 
   \curl' \psi' =  g \;\; &\text{in $ \Omega'$,}   \\
    \div' \psi' =  h  \;\; &\text{in $ \Omega'$,}    \\
     \text{$\psi'\cdot \nu'=0$  or $\psi \cdot \tau' = 0$} &\text{on $\partial \Omega'$,}  
\end{array}
\right.
\end{equation}
where $\nu'=(\nu_2,\nu_3)$ and $\tau'=(-\nu_3,\nu_2)$ denote the unit normal and tangent vectors to $\partial\Omega'$, respectively,  then
$$\norm{\nabla'\psi'}_{L^2(\Omega')}\lesssim\norm{ g }_{L^2(\Omega')}+\norm{h}_{L^2(\Omega')}.$$
\end{assumption}
This assumption guarantees the existence of the decompositions \eqref{mm:15} and \eqref{mm:16} with the spaces $\kt$ and $\kn$ (Definition \ref{def:11})
both trivial.
\begin{prop}\label{prop:cylinders}
Let $\Omega = {\mathbb R}\times\Omega'$, where $\Omega'\subseteq\R^2$ is a Lipschitz domain satisfying Assumption~\ref{ch1}. Then {$\kn=\kt=\{0\}$ and}
\begin{enumerate}[(a)]
\item $\Hdivzero = \curl\{\psi\in\dot H^1(\Omega):\div\psi=0\;\text{in $\Omega$, $\psi\cdot\nu=0$ on $\partial\Omega$}\}$,
\item $\Hdivzerozero = \curl\{\psi\in\dot H^1(\Omega):\div\psi=0\;\text{in $\Omega$, $\psi\times\nu=0$ on $\partial\Omega$}\}$.
\end{enumerate}
\end{prop}

\begin{proof}
We divide the proof into three steps.
\begin{enumerate}
\item  First, we prove that every function $f$ in $\Hdivzerozero$ may be written as the curl of a unique divergence-free function $\psi$ such that $\psi\times\nu=0$ on $\partial\Omega$. In particular, this implies that the space $\kn$  is trivial.
\item Second, we prove that every function $f$ in $\Hdivzero$ may be written as the curl of a unique divergence-free function $\psi$ such that $\psi\cdot\nu=0$ on $\partial\Omega$. In particular, this implies that the space $\kt$  is trivial.
\item Third, we prove that the potentials $\psi$ constructed in steps (1) and (2) belong to $ \dot H^1(\Omega)$.
\end{enumerate}

\textit{Step (1).}
Given $f\in \Hdivzerozero$, we look for $\psi$ such that
\begin{align} 
   \curl \psi = {}& f \;\;{\rm in}\; \Omega, \label{c1} \\
    \div \psi = {}& 0\;\;{\rm in} \;\Omega, \label{c2} \\
      \psi \times \nu = {}& 0\;\;{\rm on}\; \partial \Omega. \label{c3} 
\end{align}
Since $\nuc1=0$,  the second and third components of (\ref{c3}) yield $\psic1\nuc3=0$ and $\psic1\nuc2=0$, giving
$\psic1=0$ on $\partial \Omega$.
Taking the $\curl$ of equation (\ref{c1}) we obtain
$$-\Delta \psic1=\partialc2 f_3 -\partialc3 f_2\quad{\rm in}\; \Omega;$$
upon taking the Fourier transform with respect to the first coordinate $x_1$ we obtain the boundary value problem
\begin{equation}
\left \{  \begin{array}{ll}
 -\Delta'\psich1 +\xi^2 \psich1 =\partialc2\f{f_3}-\partialc3 \f{f_2} & {\rm in \;}\Omega',\\
 \psich1=0 &    {\rm on \;}   \partial \Omega',
 \end{array}
 \right .  \label{c4}
 \end{equation}
 in which $\Delta'$ denotes the Laplacian with respect to $x'\in\Omega'$ and $\xi\in \mathbb R$ is the dual variable of $x_1$ under Fourier 
 transformation.  {For almost every $\xi\in\R$, this Dirichlet boundary value problem admits a unique solution $\hat\psi_1(\xi)\in\dot H^1_0(\Omega')$ by the Lax Milgram theorem, and so $\psi_1$ is uniquely determined.}  To obtain the remaining components
 of $\psi$ we rewrite (\ref{c1}) and   (\ref{c2})   as 
 $$ \partialc2\psic3-\partialc3\psic2=f_1,\quad  \partialc3\psic1-\partialc1\psic3=f_2,  \quad  \partialc1\psic2-\partialc2\psic1=f_3,\quad    \partialc1 \psic1+\partialc2 \psic2+\partialc3 \psic3=0.$$   Again take the Fourier transform with respect to $x_1$ and obtain
 \begin{equation}\label{eq:c5mm1}
 \partialc2\psich3-\partialc3\psich2=\f{f_1},\quad   \partialc3\psich1-i \xi\psich3=\f{f_2},\quad   i\xi \psich2- \partialc2\psich1=\f{f_3},\quad   i\xi \psich1+\partialc2 \psich2+\partialc3 \psich3=0.\end{equation}
 Using the second and third identities in (\ref{eq:c5mm1}) yields
 $$
 \psich2=-i \dfrac{  \partialc2 \psich1+\f{f_3}}\xi, \quad  \psich3= i    \dfrac{       \f{f_2}  -\partialc3 \psich1     }{\xi },\qquad {\text{for a.e.\ $\xi\in\R$}}.
$$
 It remains to check the first and fourth identities in (\ref{eq:c5mm1}) and the first component of (\ref{c3}). 
 For the first identity in (\ref{eq:c5mm1}) we observe that
 \[ \partial_2{\hat\psi}_3-\partial_3\hat{\psi}_2 = \frac{i}{\xi}\left(\partial_2\hat{f}_2-\partial_{23}\hat{\psi}_1 + \partial_{32}\hat{\psi}_1+\partial_3 \hat{f}_3\right) = \frac{i}{\xi}(-i\xi\hat{f}_1) = \hat{f}_1.\]
Here we have used, for the second equality, the fact that $\div f=0$.  For the fourth identity in (\ref{eq:c5mm1}), by \eqref{c4} we have
\begin{align*}  i\xi \psich1+\partialc2 \psich2+\partialc3 \psich3 & =  i\xi\psich1 + \frac{i}{\xi}\left( -\partial_2^2\psich1-\partial_2\hat{f}_3+\partial_3\hat{f}_2-\partial_3^2\psich1    \right)\\
& =  \frac{i}{\xi}\left( (\xi^2\psich1-\partial_2^2\psich1-\partial_3^2\psich1) - (\partial_2\hat{f}_3-\partial_3\hat{f}_2)   \right) \\
 & =  0. \end{align*}
 Finally, for the first component of (\ref{c3}), using ${\mathcal F}$ to denote the Fourier transform,
 \[ \nu_3\psich2-\nu_2\psich3 = \frac{i}{\xi}\left(-\partial_2 \psich1 \nu_3 - \hat{f}_3\nu_3 - \hat{f}_2\nu_2 + \partial_3\psich1\nu_2 \right)
  = \frac{i}{\xi}{\mathcal F}{(\nabla' \psi_1 \cdot \tau - f\cdot \nu)} = 0,\]
  where we have used the fact that {$\psi_1=0$ and }$f\cdot \nu=0$ on $\partial\Omega$ in the last step.

 \textit{Step (2).}
 The only difference between this case and the one above lies in the boundary conditions. We no longer have $f\cdot \nu=0$ on the boundary.
This time $f\in\Hdivzero$ and we seek $\psi$ such that
 \begin{align} 
   \curl \psi = {}& f \;\;{\rm in}\;\; \Omega, \label{c11} \\
    \div \psi = {}& 0\;\;{\rm in} \;\;\Omega, \label{c12} \\
      \psi \cdot \nu = {}& 0\;\;{\rm on}\; \partial \Omega. \label{c13} 
\end{align}
The calculations follow as above except that  problem (\ref{c4}) is replaced by 
\begin{equation}
\left \{  \begin{array}{ll}
 -\Delta'\psich1 +\xi^2 \psich1 \; = \; \div'( \f{f_3},-  \f{f_2}) & {\rm in \;}\Omega',\\
 -\nabla'\psich1\cdot \nu'\; = \;  (\f{f_3},-\f{f_2})\cdot\nu' &    {\rm on \;}   \partial \Omega',
 \end{array}
 \right .  \label{c14}
 \end{equation}
 in which the reason for the slightly curious Neumann boundary condition will become clear shortly.  {As above, for almost every $\xi\in\R$, this problem admits a unique solution $\hat\psi_1(\xi)\in\dot H^1(\Omega')$ by the Lax Milgram theorem.} Having found $\psi_1$, we construct $\psich2$ and $\psich3$ as before, 
 and the verification of (\ref{c11}) and (\ref{c12}) is similar to the calculations for (\ref{c1})  and (\ref{c2}). This leaves the boundary condition (\ref{c13}): since $\hat{\psi}\cdot\nu  = \psich2\nu_2+\psich3\nu_3$, we have
\[
 \hat{\psi}\cdot\nu  = \frac{-i}{\xi}\left\{ (\partial_2\psich1+\hat{f}_3)\nu_2 - (\hat{f}_2-\partial_3\psich1)\nu_3\right\} 
 = \frac{-i}{\xi}\left\{\nabla'\psich1\cdot\nu' + (\hat{f}_3,-\hat{f}_2)\cdot \nu'\right\}  = 0, 
\]
the equality at the last step coming from the boundary equation in (\ref{c14}).
 
\textit{Step (3).}
 We now verify that  $\psi$ lies in $\dot{H}^1(\Omega)$ in both cases.  From (\ref{c4},\ref{c14})  we have 
 \begin{equation}
\left \{  \begin{array}{ll}
 -\Delta'\psich1 +\xi^2 \psich1 =\div'( \f{f_3},-  \f{f_2}) & {\rm in \;}\Omega',\\
\psich1=0\;\;{\rm or }\;\; -\nabla'\psich1\cdot \nu'= (\f{f_3},-\f{f_2})\cdot\nu' &    {\rm on \;}   \partial \Omega'.
 \end{array}
 \right .  \label{c15}
 \end{equation}
 An integration  against $\psich1$ gives
 $$ \norm{\nabla'\psich1}^2_{L^2(\Omega')} +  \xi^2 \norm{\psich1}_{L^2(\Omega')}^2=
 ((-\f{f_3},\f{f_2}),\nabla'\psich1)_{L^2(\Omega')}
 $$
 whence
 \begin{equation}
 \norm{ \nabla'\psich1(\xi)}_{L^2(\Omega')}  
 \leq\norm{     \f{  f'}(\xi)  }_{L^2(\Omega')}, \;\;\;
 \norm{  \psich1(\xi)}_{L^2(\Omega')}\leq \dfrac{   \norm{   \f{  f'} (\xi) }_{L^2(\Omega')}}{|\xi|}.
 \label{c7} 
 \end{equation}
From the first and fourth identities of \eqref{eq:c5mm1}
 we get, for {almost every} $\xi \in \R$,
\[
   \begin{array}{ll} 
   \curl' \hat{\psi}' = {} \hat{f}_1  \;\; &{\rm in} \;\;   \Omega',   \\
    \div' \hat{\psi}' = {} -i\xi\psich1  \;\; &{\rm in} \;\;  \Omega'.   
\end{array}
\]
We also have the desired boundary conditions:
$$\hat \psi' \cdot \tau'=0  \;{\rm on} \; \partial \Omega'\; \mbox{for (b),}\;\; {\rm or } \;\;  \hat \psi' \cdot \nu'=0\; {\rm on} \; \partial \Omega'\; \mbox{for (a).}
$$
By Assumption~\ref{ch1}  we have  
\begin{equation}\label{c5mm3}\norm{\nabla'\hat \psi'(\xi)}_{L^2(\Omega')}   \lesssim \norm{\hat{f}_1(\xi)}_{L^2(\Omega')}+\norm {\xi\psich1(\xi)}_{L^2(\Omega')} \lesssim \norm{\hat{f} (\xi)}_{L^2(\Omega')},
\end{equation}
the last inequality following from the second inequality in (\ref{c7}). 

 We now regularise $\hat\psi_i$  for $\xi\to 0$.
 Define, for $\epsilon >0$,
 $$\f{\psi}_{i,\epsilon}(\xi,x')=\dfrac{|\xi|}{|\xi|+\epsilon} \f{\psi_{i}}(\xi,x'),\quad i=1,2,3.$$
 By a direct calculation we have
 $\nabla'\f{\psi_i}-\nabla'\f{\psi}_{i,\epsilon}   =\dfrac{\epsilon}{|\xi| +\epsilon}\nabla'\f{\psi}_i$,
 and so we get
 $$
 \norm{  \nabla'\f{\psi_i}-\nabla'\f{\psi}_{i,\epsilon}}_{L^2(\Omega')} =\dfrac{\epsilon}{|\xi|+\epsilon}    \norm{  \nabla'\f{\psi}_i}_{L^2(\Omega')} \leq \norm{  \nabla'\f{\psi}_i}_{L^2(\Omega') } \lesssim
 \norm{  \f{f}(\xi)}_{L^2(\Omega') } ,$$
 where the last inequality follows from (\ref{c7}) for $i=1$ and (\ref{c5mm3})  for $i=2,3$.  
 By the Dominated Convergence Theorem, therefore, 
 $$\lim_{\epsilon\to 0}\norm{\nabla'\f{\psi_{i}}-\nabla'\f{\psi}_{i,\epsilon}}_{L^2(\Omega)}= 0.$$
By direct calculation,
 \begin{equation}\label{c5mm4}\norm{  \xi \f{\psi}_{i}-\xi \f{\psi}_{i,\epsilon}}_{L^2(\Omega')}=\dfrac{ |\xi |\epsilon}{|\xi|+\epsilon} \norm{  \f{\psi}_{i}}_{L^2(\Omega')}.
 \end{equation}
 Now, using (\ref{c7}) for $i=1$  and the two identities
 $$\xi\f{\psi}_2=-i(\partialc2\f{\psi}_1+\f{f}_3), \quad   \xi\f{\psi}_3 = i  (\f{f}_2-\partialc2\f{\psi}_1),
 $$
 together with (\ref{c7}) for $i=2,3$, we have $|\xi|\cdot\norm{\f{\psi}_i}_{L^2(\Omega') }\lesssim\norm{\f{f}}_{L^2(\Omega') } $, whence
 $\norm{\xi \f{\psi_i}-\xi \f{\psi_{i,\epsilon}}}_{L^2(\Omega')}\lesssim\norm{\f{f}}_{L^2(\Omega')}.
 $
 Taking inverse Fourier transforms in (\ref{c5mm4}), dominated convergence yields
 $$\lim_{\epsilon\to 0}\norm{  \partial_1\psi_{i,\epsilon}-\partial_1 \psi_{i }}_{L^2(\Omega)}
= 0 . $$
Altogether we have that $ \norm{  \nabla\psi_{i,\epsilon}- \nabla\psi_{i }}_{L^2(\Omega)}\to 0$ as $\epsilon\to 0$, namely
\begin{equation}
\lim_{\epsilon\to 0}\norm{  \psi_{i,\epsilon}- \psi_{i }}_{\dot H^1(\Omega)}= 0.  \label{c8}
\end{equation}

Since  $| \f{\psi}_{i,\epsilon}(\xi,x')| \leq \epsilon^{-1} | \xi \f{\psi_{i }}(\xi,x')| $   we have that 
$ \norm{ \psi_{i,\epsilon}}_{L^2(\Omega)}\leq \epsilon^{-1}  \norm{ \partialc1\psi_i}_{L^2(\Omega)}<+\infty$,
and so   $ \psi_{i,\epsilon}\in H^1(\Omega)$.   Since $H^1(\Omega)$ is dense in $\dot{H}^1(\Omega)$,  
by (\ref{c8}) we conclude that $\psi\in \dot{H}^1(\Omega)$.
\end{proof}

{
We now observe that Assumption~\ref{ch1} is verified in many situations of interest.
\begin{lem}\label{lem:cylinders}
Assumption~\ref{ch1} is verified in each of the following cases:
\begin{enumerate}
\item $\Omega'$ is the full space $\R^2$;
\item $\Omega'$ is the half space $\{(x_2,x_3)\in \R^2: x_3>0\}$;
\item $\Omega'$ is a strip $\{(x_2,x_3)\in \R^2: 0<x_3<L\}$ for some $L>0$;
\item $\Omega'$ is a  simply connected bounded domain of class 
$C^{1,1}$ or piecewise smooth with no re-entrant corners.
\end{enumerate}
\end{lem}
}
\begin{proof}
\textit{(1)}
Taking Fourier transforms in \eqref{eq:hp1} we obtain
$$  i\xi_2 \f{\psi}_3'-i\xi_3 \f{\psi}_2'=\f{g},\quad i\xi_2 \f{\psi}_2'+i\xi_3 \f{\psi}_3'=\f{h},$$
with unique solution
$$\f{\psi}_2'= \dfrac{ -i \xi_2 \f{h}+i\xi_3 \f{g}}{|\xi'|^2}, \quad
\f{\psi}_3'= \dfrac{  -i\xi_3 \f{h}-i\xi_2 \f{g}}{|\xi'|^2}.$$
Hence
$|\xi_2\f{\psi_2'}|= \frac{  |\xi_2^2\f{h}|}{|\xi'|^2}  +    \frac{  |\xi_2  \xi_3 \f{g}|}{|\xi'|^2}\leq  |\f{h}|+   |\f{g}|,$ so that
$$\norm{ \partial_2 {\psi_2'}}_{L^2(\R^2)}  \leq\norm{\f{g}}_{L^2(\R^2)}+ \norm{\f{h}}_{L^2(\R^2)},$$
and similarly  for the other conditions.

\textit{(2)} In this case
the boundary condition is either $\psi'_2=0$  or $\psi'_3=0$ on $\{x_3=0\}$. We study the case $\psi'_2=0$; the other is similar.

Taking Fourier transforms  we get
\begin{align*}
  \psi_2'(x_2,x_3)&=\int_\R\int_{\R_+}  \f{\psi}'_2 (\xi_2,\xi_3)e^{i\xi_2x_2} \sin(\xi_3 x_3)d \xi_3 d\xi_2,\\
 \psi_3'(x_2,x_3)&=\int_\R\int_{\R_+}  \f{\psi}'_2 (\xi_2,\xi_3)e^{i\xi_2x_2} \cos(\xi_3 x_3)d \xi_3 d\xi_2,\\
 g(x_2,x_3)&=\int_\R\int_{\R_+}  g (\xi_2,\xi_3)e^{i\xi_2x_2} \cos(\xi_3 x_3)d \xi_3 d\xi_2,\\
 h(x_2,x_3)&=\int_\R\int_{\R_+}  h (\xi_2,\xi_3)e^{i\xi_2x_2} \sin(\xi_33 x_3)d \xi_3 d\xi_2.
\end{align*}
The equations  in \eqref{eq:hp1} become
$$i\xi_2\f{\psi_3'}-\xi_3\f{\psi_2'}=\f{g},\qquad  i\xi_2\f{\psi_2'}-\xi_3\f{\psi_3'}=\f{h},
$$
and then everything proceeds as for the case $\Omega'={\mathbb R}^2$.

\textit{(3)} This follows by using the Fourier transform with respect to the variable $x_2$ and the Fourier series in the variable $x_3$, as in cases (1) and (2)
above; the calculations are completely analogous.

\textit{(4)}
This part was proven in \cite[Chapter~1, Remark~3.5]{GIRAULT-RAVIART-1986}).
\end{proof}

\bibliographystyle{abbrv}
\bibliography{biblio}

\end{document}